\RequirePackage{luatex85}
\documentclass[11pt]{amsart}
\usepackage{etex}

\usepackage[margin=1in]{geometry}

\usepackage{amscd,latexsym,amsthm,amsfonts,amssymb,amsmath,amsxtra,mathdots,tikz}
\usepackage[mathscr]{eucal}
\usepackage{xcolor}
\usepackage[normalem]{ulem}
\usepackage{hyperref}
\usepackage[all,cmtip]{xy}
\usepackage{enumitem}
\usepackage{tabularx}
\usepackage{textcomp}
\usepackage{caption}
\usepackage{makecell}
\usepackage{pdflscape}

\renewcommand\theequation{\thesection.\arabic{equation}}

\newcommand{\BG}{{\mathbb {G}}}

\newcommand{\CM}{{\mathcal {M}}}

\newcommand{\Fe}{{\mathfrak {e}}}
\newcommand{\Ff}{{\mathfrak {f}}}
\newcommand{\Fg}{{\mathfrak {g}}}
\newcommand{\Fh}{{\mathfrak {h}}}

\newcommand{\Fl}{{\mathfrak {l}}}

\newcommand{\Ft}{{\mathfrak {t}}}
\newcommand{\Fu}{{\mathfrak {u}}}

\newcommand{\Fgl}{\mathfrak{gl}}
\newcommand{\Fsl}{\mathfrak{sl}}
\newcommand{\Fsp}{\mathfrak{sp}}
\newcommand{\Fso}{\mathfrak{so}}
\newcommand{\Fspin}{\mathfrak{spin}}

\newcommand{\GL}{{\mathrm{GL}}}

\newcommand{\GSp}{{\mathrm{GSp}}}

\newcommand{\GSO}{{\mathrm{GSO}}}

\newcommand{\Spin}{{\mathrm{Spin}}}

\newcommand{\PGL}{{\mathrm{PGL}}}

\newcommand{\SL}{{\mathrm{SL}}}

\newcommand{\HSpin}{{\mathrm{HSpin}}}

\newcommand{\SO}{{\mathrm{SO}}}

\newcommand{\Sp}{{\mathrm{Sp}}}

\def\Fu{{\mathfrak u}}

\newtheorem{thm}{Theorem}[section]

\newtheorem{prop}[thm]{Proposition}
\newtheorem {conj}[thm]{Conjecture}

\newtheorem {ques/conj}[thm]{Question/Conjecture}

\newtheorem{defn}[thm]{Definition}
\newtheorem{rmk}[thm]{Remark}

\newtheorem{lemma}[thm]{Lemma}

\makeatletter

\newcommand{\Rmnum}[1]{\expandafter\@slowromancap\romannumeral #1@}
\makeatother

\begin{document}
	\renewcommand{\theequation}{\arabic{equation}}
	\numberwithin{equation}{section}

	\title{Anomaly-free Hyperspherical Hamiltonian spaces for simple reductive groups}

	\author{Guodong Tang}
	\address{Department of Mathematics\\
		National University of Singapore, Singapore}
	\email{E1124897@u.nus.edu}

	\author{Chen Wan}
	\address{Department of Mathematics \& Computer Science\\
		Rutgers University – Newark\\
		Newark, NJ 07102, USA}
	\email{chen.wan@rutgers.edu}
	
	\author{Lei Zhang}
	\address{Department of Mathematics\\
		National University of Singapore, Singapore}
	\email{matzhlei@nus.edu.sg}
	
	\date{}
	
	

	\begin{abstract}
		In this paper, we provide a complete list of anomaly-free hyperspherical Hamiltonian spaces for simple reductive groups, as well as their conjectural dual spaces in the sense of BZSV duality \cite{BSV}. 
	\end{abstract}
	
	\maketitle

\section{Introduction}
Let $G$ be a split connected reductive group and $\hat{G}$ be its dual group. Following Section 3.5 of \cite{BSV}, we say that a smooth affine $G$-Hamiltonian space $\CM$ is hyperspherical if it satisfies the following three conditions:
\begin{itemize}
\item (Coisotropic condition) The field of $G$-invariant rational functions on $\CM$ is commutative with respect to the Poisson bracket.
\item The image of the moment map $\CM\rightarrow \Fg^\ast$ has a nonempty intersection with the nilcone of $\Fg^\ast$.
\item The stabilizer in $G$ of a generic point of $\CM$ is connected.
\end{itemize}

In Section 3.6 of \cite{BSV}, Ben-Zvi, Sakellaridis, and Venkatesh proved a structure theorem for those Hamiltonian space which we   recall here. We define a BZSV quadruple for $G$ to be $\Delta=(G,H,\iota,\rho_H)$ where $H$ is a split reductive subgroup of $G$;  $\rho_H$  is a symplectic representation of $H$; and  $\iota$ is a homomorphism from $\SL_2$ into $G$ whose image commutes with $H$. For a BZSV quadruple $\Delta=(G,H,\iota,\rho_H)$ of $G$, following \cite[Section 3]{BSV}, one can associate a $G$-Hamiltonian variety $\CM_\Delta$ as follows. 
Let $L$ be the centralizer of $h(t):=\iota(\begin{pmatrix}t&0\\ 0&t^{-1}\end{pmatrix})$ in $G$  and let $U=\exp(\mathfrak u)$ (resp. $\bar{U}=\exp(\bar{\mathfrak u})$) be the corresponding unipotent subgroups of $G$ associated with $\iota$, where $\mathfrak u\subset \Fg$ (resp. $\bar{\mathfrak u}\subset \Fg$) is the positive weight space (resp. negative weight space) of the Lie algebra $\Fg$ of $G$ under the adjoint action of $h(t)$. Then $P=LU$ and $\bar{P}=L\bar{U}$ are parabolic subgroups of $G$ that are opposite to each other. Since $H$ commutes with the image of $\iota$, we have $H\subset L$.

Let $\mathfrak u^+$ be the $\geq 2$ weight space under the adjoint action of $h(t)$. It is well known that the vector space $\mathfrak u/\mathfrak u^+$ has a symplectic structure and realizes a symplectic representation of $H$ under the adjoint action. If we denote by $V=V_{\rho_H}$ the underlying vector space of the symplectic representation $\rho_H$, then the Hamiltonian variety $\CM_\Delta$ is defined as
\begin{equation}\label{HV}
 \CM_\Delta=((V\times \mathfrak{u}/\mathfrak{u}^+)\times _{(\Fh+\mathfrak{u})^\ast} \Fg^\ast)\times ^{HU}G   
\end{equation}
with the following structures:
\begin{itemize}
\item the maps $\mathfrak{u}/\mathfrak{u}^+\rightarrow \Fh^\ast$ and $V\rightarrow \Fh^\ast$ are the moment maps;
\item the map $V\rightarrow \mathfrak{u}^\ast$ is the zero map;
\item the map $\mu:\mathfrak{u}/\mathfrak{u}^+\rightarrow \mathfrak{u}^\ast$ is given by
$$\mu(u)=\kappa_1(u)+\kappa_f$$
where $\kappa_1: \mathfrak{u}/\mathfrak{u}^+\rightarrow (\mathfrak{u}/\mathfrak{u}^+)^\ast$ is the isomorphism via the symplectic form on $\mathfrak{u}/\mathfrak{u}^+$ and 
$$\kappa_f(X)=(f,X),\;X\in \mathfrak{u},\;f=\iota(\begin{pmatrix} 0&0\\1&0\end{pmatrix}).$$
\end{itemize}

Theorem 3.6.1 of \cite{BSV} states that any hyperspherical $G$-Hamiltonian space is of the form $\CM_\Delta$ associated with a unique BZSV quadruple $\Delta=(G,H,\iota,\rho_H)$ of $G$.

\begin{defn}
For a BZSV quadruple $\Delta=(G,H,\iota,\rho_H)$, we use $\rho_{\iota}$ to denote the symplectic representation $\Fu/\Fu^+$ of $H$ and let $\rho_{H,\iota}=\rho_H\oplus \rho_{\iota}$. We say the quadruple is hyperspherical if the associated Hamiltonian space $\CM_\Delta$ is hyperspherical. 
\end{defn}

The following proposition has been proved in Section 3.6 of \cite{BSV}.
\begin{prop}\label{hyperspherical prop}
A BZSV quadruple $\Delta=(G,H,\iota,\rho_H)$ is hyperspherical if it satisfies the following conditions:
\begin{itemize}
\item \label{item:Condition-1} $H$ is a spherical subgroup of $L$.
\item \label{item:Condition-2} The restriction of the symplectic representation $\rho_{H,\iota}$ to $H'$ is a multiplicity-free symplectic representation (completely classified by Knop \cite{K} and Losev \cite{Lo}) where $H'$ is the stabilizer (in $H$) of a generic point of $\Fh^\perp$. Here $\Fh^\perp$ is the orthogonal complement of $\Fh$ in $\Fl$ under the Killing form and $H$ acts on it by the conjugation action.
\item The stabilizer (in $G$) of a generic point of $\CM_\Delta$ is connected.
\end{itemize}
\end{prop}

\begin{defn}
We say a BZSV quadruple $\Delta=(G,H,\iota,\rho_H)$ is anomaly-free if the symplectic representation $\rho_{H,\iota}$ is an anomaly-free symplectic representation of $H$ (defined in Definition 5.1.2 of \cite{BSV}). We say a hyperspherical Hamiltonian space $\CM=\CM_{\Delta}$ is anomaly-free if $\Delta$ is anomaly-free.
\end{defn}

In \cite{BSV}, Ben-Zvi--Sakellaridis--Venkatesh proposed a conjectural duality between the set of anomaly-free hyperspherical $G$-Hamiltonian spaces and the set of  anomaly-free hyperspherical $\hat{G}$-Hamiltonian spaces, or equivalently, a conjectural duality between the set of anomaly-free hyperspherical BZSV quadruples of $G$ and the set of anomaly-free hyperspherical BZSV quadruples of $\hat{G}$. 
This proposed duality not only extends the classical Langlands program to a broader geometric setting but also provides a new perspective on the interaction between Hamiltonian symmetries and representation theory. 
They also formulated a series of elegant and far-reaching conjectures that should hold within this framework. 

An important aspect of their conjecture concerns period integrals. Roughly speaking, their period conjecture states that the period integral associated with $\Delta$ should be equal to the L-function  associated with $\hat{\Delta}$, and vice versa. Here $\Delta$ and $\hat{\Delta}$ are any two anomaly-free hyperspherical BZSV quadruples that are dual to each other. We refer the reader to Section 1 of \cite{MWZ1} for a detailed statement of the period integral conjecture. The period integral conjecture not only gives a conceptual explanation for many established automorphic integrals, but also introduces many new classes of automorphic integrals for studying.

Despite its conceptual beauty, a major challenge in BZSV duality is the lack of a general algorithm to compute the duality.  In other words, for a given anomaly-free hyperspherical BZSV quadruple $\Delta=(G,H,\rho_H,\iota)$, there is currently no known systematic procedure to determine its dual $\hat{\Delta}$. This remains a fundamental open problem. There is also no general classification of anomaly-free hyperspherical BZSV quadruples for general reductive groups.
	
In Section 4 of \cite{BSV}, the authors devised an algorithm to compute the dual in a special case known as the polarized case, which is when the symplectic representation $\rho_{H,\iota}$ of $H$ is of the form $\rho_{H,\iota}=T(\tau):=\tau\oplus \tau^\vee$ for some representation $\tau$ of $H$. In particular, this includes the cases where $\Delta=(G,H,0,1)$ (i.e. the spherical variety case). In a paper by Mao, the second and third authors \cite{MWZ2}, they provide an algorithm to compute the dual in the vector space case, i.e., the case where $\Delta=(G,G,\rho,1)$.

In this paper, we study the case where $G$ is a simple reductive group. We provide a complete classification of anomaly-free hyperspherical BZSV quadruples for any simple reductive group, along with their dual quadruples. Our findings are summarized in the tables provided at the end of the paper. As in \cite{MWZ2}, by examining the period integral conjecture for these quadruples, one can recover numerous previously studied Rankin-Selberg and period integrals, thereby providing a new conceptual framework for understanding them. Furthermore, this work also introduces several new period integrals for future study. We do not explore these applications further within the scope of this paper.

\subsection{Organization of the paper}
In Section \ref{sec:strategy}, we explain our strategy for classifying the quadruple and computing dual quadruples, using Type $B_2=C_2$ and $G_2$ as illustrative examples. Section \ref{sec:type-A} and Section \ref{sec:typeE8} treat cases of Type $A$ and Type $E_8$, respectively.

The arguments for the remaining types follow a similar logic: the classical cases (Types $B, C, D$) are analogous to Type $A$, while the exceptional cases (Types $F_4, E_6, E_7$) parallel the treatment of Type $E_8$. Consequently, we omit the explicit details for these instances and instead summarize the results in the tables provided in Section \ref{sec:tables}. We refer the reader to the PhD thesis of the first author for the details of those cases.

\subsection{Acknowledgement}
The second author’s work is partially supported by the NSF grant DMS-2349836 and a Simons Travel Grant. The work of the third author is partially supported by AcRF Tier 1 grants A-0004274-00-00, A-0004279-00-00, and A-8002960-00-00 of the National University of Singapore.

\section{Our strategy} \label{sec:strategy}
In this section we will explain our strategy to classify the quadruples and to compute the dual. We will use Type $B_2=C_2$ and $G_2$ as examples to explain our strategy. In Section \ref{sec:strategy-classification}, we will explain our strategy for classifying the quadruples. In Section \ref{sec:strategy-dual}, we will explain our strategy for computing the dual.

\subsection{Strategy for classifying the quadruples} \label{sec:strategy-classification}
Let $G$ be a split connected reductive group. We will explain our strategy for classifying anomaly-free hyperspherical BZSV quadruples of $G$. The first step is to identify all the $\SL_2$-homomorphisms $\iota:\SL_2\rightarrow G$ (or equivalently, all the nilpotent orbits) that can appear in a quadruple. For any given $\iota$,   let $L,U,\bar{U},\Fu,\Fu^+$ be as in the introduction and let $G_\iota\subset L$ be the centralizer of $Im(\iota)$ in $G$. 
If $\iota$ is part of a hyperspherical BZSV quadruple $\Delta=(G,H,\iota,\rho_H)$ of $G$, then $H\subset G_\iota$. By Proposition \ref{hyperspherical prop}, we must have
\begin{enumerate}
\item\label{condition 1} $G_\iota$ is a spherical subgroup of $L$. 
\item\label{condition 2} $\Fu/\Fu^+$ is a multiplicity-free symplectic representation of $G_\iota'$ where $G_\iota'$ is the stabilizer (in $G_\iota$) of a generic point of $\Fg_{\iota}^{\perp}$. Here $\Fg_{\iota}^{\perp}$ is the orthogonal complement of $\Fg_\iota$ in $\Fl$.
\end{enumerate}
We use $Nil_{0}(\Fg)$ to denote the subset of all nilpotent orbits that satisfies the two conditions above. 
We refer the reader to \cite{BP} (resp. \cite{K}) for a complete list of spherical reductive subgroups (resp. multiplicity-free symplectic representations). When $G$ is a classical group, Gan and Wang provide a classification of all the even nilpotent orbits $\iota$ \footnote{here we say $\iota$ is even if $\Fu=\Fu^+$, in particular Condition \eqref{condition 2}  is trivial if $\iota$ is even} that satisfy the two conditions above. 
They also give a necessary condition for general $\iota$ in the classical group case.

Throughout this paper, for nilpotent orbits, $1$ always stands for the trivial nilpotent orbit. In the classical groups case, we will use partitions to denote the nilpotent orbits. In the exceptional group case, we will the Levi subgroup to denote the nilpotent orbit (i.e. it denotes the nilpotent orbit that is principal in the Levi), and we will see that all the nilpotent orbits in $Nil_0(\Fg)$ are principal in a Levi.

The next step is to determine all anomaly-free hyperspherical BZSV quadruples $\Delta=(G,H,\iota,\rho_H)$ that contain a given $\iota\in Nil_0(\Fg)$. 
We fix such an $\iota$, which gives us $L,G_\iota\subset L,\Fu/\Fu^+$. 
Then the group $H$ in the quadruple must be a spherical subgroup of $G$ contained in $G_\iota$ such that $\Fu/\Fu^+$ is a multiplicity-free symplectic representation of $H'$ where $H'$ is the stabilizer (in $H$) of a generic point of $\Fh^{\perp}$ and $\Fh^{\perp}$ is the orthogonal complement of $\Fh$ in $\Fl$. 
One can systematically examine all the spherical subgroups of $L$ (listed in \cite{BP}) that satisfy these conditions. 

The final datum to be determined is the symplectic representation $\rho_H$ of $H$. For this, we just need to go over all multiplicity-free symplectic representations of $H$ (listed in \cite{K}) such that the restriction of $\rho_H\oplus (\Fu/\Fu^+)$ to $H'$ is still multiplicity-free. This is how we write down all the BZSV quadruples that satisfy the first two conditions in Proposition \ref{hyperspherical prop}. Lastly, we just need to make sure the representation $\rho_H\oplus (\Fu/\Fu^+)$ of $H$ is anomaly-free and the stabilizer of a generic point of $\CM_\Delta$ in $G$ is connected. This completes the classification.

\begin{rmk}\label{rmk isogeny}
For a BZSV quadruple $\Delta=(G,H,\iota,\rho_H)$, there exist various other quadruples that are essentially equivalent to $\Delta$ up to a central or finite isogeny. For instance, the quadruples $(\GL_2,\SL_2,1,T(std))$ and $(\SL_2,\SL_2,1,T(std))$  are fundamentally the same. In this paper, we provide only one representative for each such family (i.e., one for each root type).
\end{rmk}

For the rest of this subsection, we will discuss the cases of Type $B_2=C_2$ and $G_2$ as examples. We start with the $B_2=C_2$ case. In this case, $G=\Sp_4$ and $Nil(\Fg)=\{(4),(2^2),(2,1^2),(1^4)\}$ contains 4 elements. It is clear that $(4)$ and $(1^4)$ belong to the set $Nil_0(\Fg)$ (in fact, it is clear that the trivial orbit and the regular orbit belong to $Nil_0(\Fg)$ for any $G$). For $(2^2)$, $L=\GL_2$, $G_\iota=O(2)$ and $\Fu/\Fu^+=\{0\}$. For $(2,1^2)$, $L=\GL_1\times \Sp_2$, $G_\iota=\Sp(2)$ and $\Fu/\Fu^+$ is the standard representation of $G_\iota$. It is easy to see that both orbits satisfy Conditions \eqref{condition 1}  and \ref{condition 2}  above. Hence $Nil_0(\Fg)=Nil(\Fg)$. Now we will determine all the possible $H$ and $\rho_H$ for each $\iota$.

\begin{itemize}
\item When $\iota$ is the trivial orbit, $H$ can be all the possible spherical subgroups of $\Sp_4$, which are $\Sp_4,\Sp_2\times \Sp_2,\GL_1\times \Sp_2$ or $\GL_2$. 
\begin{itemize}
    \item When $H=\Sp_4$, $\rho_H$ can be all the anomaly-free multiplicity-free symplectic representations of $H$, which are $0,T(std)$ or $T(std\oplus std)$.
    \item When $H=\GL_1\times \Sp_2$, the generic stabilizer is trivial and hence $\rho_H$ must be $0$.
    \item When $H=\Sp_2\times \Sp_2$, the generic stabilizer is $H'=\Sp_2$. It is easy to see that the only possible $\rho_H$ are $0,T(std_{\SL_2}), T(std_{\SL_2,1}\oplus std_{\SL_2,2})$ or $T(std_{\SL_2,1}\oplus std_{\SL_2,1})$. Here $std_{\SL_2,i}$ means the standard representation of the $i$-th copy of $\SL_2$ in $H$.
    \item When $H=\GL_2$, $H'$ is a finite group, which forces $\rho_H$ to be zero. Then it is easy to see that the generic stabilizer on the Hamiltonian space is not connected.
\end{itemize}
\item When $\iota=(2,1^2)$, $L=\GL_1\times \Sp_2$, $G_\iota=\Sp_2$ and $\Fu/\Fu^+$ is the standard representation of $\Sp_2=\SL_2$. The only possible $H$ are $G_\iota$ or $\GL_1$. If $H=\GL_1$, $H'$ would be trivial; 
this is not hyperspherical since $\Fu/\Fu^+$ is non-zero. Hence $H$ must be $G_\iota$. The generic stabilizer of $H$ in $L$ is $\GL_1$ and hence the only possible $\rho_H$ is $std_{\SL_2}$ (note that we cannot let $\rho_H=0$ because we need $\rho_{H,\iota}$ to be anomaly-free).
\item When $\iota=(2^2)$, $L=\GL_2$ is the Siegel Levi of $G$, $G_\iota=\SO(2)=\GL_1$ and $\Fu/\Fu^+=0$. In this case, the only possible $H$ is $G_\iota$ and $\rho_H$ must be 0. This gives the quadruple $(\Sp_4,\GL_1,(2^2),0)$.
\item When $\iota$ is the regular orbit, it is clear that $H$ can only be $1$ and $\rho_H$ must be $0$ \footnote{this applies to the regular nilpotent orbit of any group}. This gives the quadruple $(\Sp_4,1,(4),0)$.
\end{itemize}

In summary, we get the following table of anomaly-free hyperspherical BZSV quadruples for Type $B_2=C_2$. For each of them, we list $G, H,\iota, \rho_H,L,\Fu/\Fu^+,H'$ and the restriction of $\rho_{H,\iota}$ to $H'$.

\newpage

\begin{figure}[h!]
	\begin{tabular}{| c | c | c | c | c | c |}
		\hline
		\textnumero & $\Delta=(G,H,\iota,\rho_H)$ & $L$ & $\Fu/\Fu^+$ & $H'$ & $\rho_{H,\iota}|_{H'}$  \\
		\hline
		1 &  $(\Sp_4,\Sp_4,1,0)$ & $\Sp_4$ & $0$ & $\Sp_4$ & $0$   \\
		\hline
        2 & $(\Sp_4,\Sp_4,1,T(std_{\Sp_4}))$ & $\Sp_4$ & $0$ & $\Sp_4$ & $T(std_{\Sp_4})$  \\
        \hline
        3 & $(\Sp_4,\Sp_4,1,T(std_{\Sp_4}\oplus std_{\Sp_4}))$ & $\Sp_4$ & $0$ & $\Sp_4$ & $T(std_{\Sp_4}\oplus std_{\Sp_4}))$  \\
        \hline
        4 & $(\Sp_4,\GL_1\times \Sp_2,1,0)$ & $\Sp_4$ & $0$ & $1$ & $0$ \\
        \hline
        5 & $(\Sp_4,\Sp_2\times \Sp_2,1,0)$ & $\Sp_4$ & $0$ & $\Sp_2$ & $0$ \\
        \hline
        6 & $(\Sp_4,\Sp_2\times \Sp_2,1,T(std_{\SL_2}))$ & $\Sp_4$ & $0$ & $\Sp_2$ & $T(std_{\SL_2}))$ \\
        \hline
        7 & $(\Sp_4,\Sp_2\times \Sp_2,1,T(std_{\SL_2,1}\oplus std_{\SL_2,2}))$ & $\Sp_4$ & $0$ & $\Sp_2$ & $T(std_{\SL_2}\oplus std_{\SL_2})$ \\
        \hline
        8 & $(\Sp_4,\Sp_2\times \Sp_2,1,T(std_{\SL_2,2}\oplus std_{\SL_2,2}))$ & $\Sp_4$ & $0$ & $\Sp_2$ & $T(std_{\SL_2}\oplus std_{\SL_2})$ \\
        \hline
        9 & $(\Sp_4,1,(4),0)$ & $\GL_{1}^2$ & $0$ & $1$ & $0$ \\
        \hline
        10 & $(\Sp_4,\GL_1,(2^2),0)$ & $\GL_2$ & $0$ & $1$ & $0$ \\
        \hline
        11 & $(\Sp_4,\SL_2,(2,1^2),std_{\SL_2})$ & $\GL_2$ & $std_{\SL_2}$ & $\SL_2$ & $T(std_{\SL_2})$ \\
        \hline
	\end{tabular}
	\captionof{table}{Type $B_2=C_2$}
    \label{Table B_2 1}
\end{figure}

For the $G_2$-case, $G=G_2$ and $Nil(\Fg)$ contain 5 elements: the regular orbit, the subregular orbit, the orbit associated with the simple root $\alpha$ (resp. $\beta$), and the trivial orbit. Hence $\alpha$ is the long root and $\beta$ is the short root. It is clear that the trivial orbit and the regular orbit belong to $Nil_0(\Fg)$. For the subregular orbit, $L=\GL_2$ and $H=1$ which is not a spherical subgroup. Hence it does not belong to $Nil_0(\Fg)$. For the orbit associated with $\alpha$ (resp. $\beta$), $L=\GL_2$, $H=\SL_2$ and $\Fu/\Fu^+$ is the symmetric cube (resp. standard representation) of $\SL_2$. It is clear that both orbits satisfy Conditions \eqref{condition 1}  and \eqref{condition 2}  above. Hence $Nil_0(\Fg)$ contains 4 orbits: the trivial orbit, the regular orbit, and the two orbits associated with two simple roots. Now we will determine all the possible $H$ and $\rho_H$ for each $\iota$.

\begin{itemize}
\item When $\iota$ is the trivial orbit, $H$ can be all the possible spherical subgroups of $G_2$, which are $G_2,\SL_3$ or $\SL_2\times \SL_2$. 
\begin{itemize}
    \item When $H=G_2$, $\rho_H$ can be all the anomaly-free multiplicity-free symplectic representations of $H$, which are $0$ and $T(std)$. If $\rho_H=T(std)$ the generic stabilizer is not connected; hence, it has to be $0$.
    \item When $H=\SL_3$, the generic stabilizer $H'$ is $\SL_2$. Then the only possible $\rho_H$ is just $0$.
    \item When $H=\SL_2\times \SL_2$, $H'$ is a finite group, which forces $\rho_H$ to be zero. Then it is easy to see that the generic stabilizer on the Hamiltonian space is not connected.
\end{itemize}
\item When $\iota$ is the orbit associated with $\alpha$, $L=\GL_2$, $H=H'=\SL_2$ and $\Fu/\Fu^+$ is the symmetric cube representation of $H$. Then $\rho_H$ must be $0$. It is easy to see that the generic stabilizer on the Hamiltonian space is not connected in this case.
\item  When $\iota$ is the orbit associated with $\alpha$, $L=\GL_2$, $H=H'=\SL_2$ and $\Fu/\Fu^+$ is the standard representation of $H$. In this case the only possible $\rho_H$ is the standard representation.
\item When $\iota$ is the regular orbit, it is clear that $H$ can only be $1$ and $\rho_H$ must be $0$.
\end{itemize}

In summary, we get the following table of anomaly-free hyperspherical BZSV quadruples for Type $G_2$.

\begin{figure}[h!]
	\begin{tabular}{| c | c | c | c | c | c |}
		\hline
		\textnumero & $\Delta=(G,H,\iota,\rho_H)$ &  $L$ & $\Fu/\Fu^+$ & $H'$ & $\rho_{H,\iota}|_{H'}$  \\
		\hline
		1 &  $(G_2,G_2,1,0)$ & $G_2$ & $0$ & $G_2$ & $0$   \\
		\hline
        2 & $(G_2,\SL_3,1,0)$ & $G_2$ & $0$ & $\SL_2$ & $0$ \\
       \hline 
       3& $(G_2, \SL_2, \SL_2, std)$ & $\GL_2$ & $std_{\SL_2}$ & $\SL_2$ & $T(std_{\SL_2})$ \\
       \hline
       4 & $(G_2,1,G_2,0)$ & $\GL_{1}^{2}$ & $0$ & $1$ & $0$ \\
       \hline
	\end{tabular}
	\captionof{table}{Type $G_2$}
    \label{Table G_2 1}
\end{figure}

\subsection{Strategy for computing the dual} \label{sec:strategy-dual}
In this subsection we will explain our strategy to compute the dual of a quadruple $\Delta=(G,H,\iota,\rho_H)$. For each $\Delta$, we have defined $L,H,\rho_{H,\iota}$ in the introduction. We set $\Delta_{red}=(L,H,1,\rho_{H,\iota})$ and we say $\Delta$ is reductive if $\iota=1$ ($\iff \Delta=\Delta_{red}$).

In Section 4.2.2 of \cite{BSV}, Ben-Zvi, Sakellaridis, and Venkatesh conjectured a relation between the dual of $\Delta$ and $\Delta_{red}$. To state their conjecture, we first need a definition.
		
\begin{defn}
Let $L$ be a Levi subgroup of $G$ and $\rho$ be an irreducible representation of $L$ with the highest weight $\varpi_L$. There exists a Weyl element $w$ of $G$ such that $w\varpi_L$ is a dominant weight of $G$ \footnote{the choice of $w$ is not unique but $w\varpi_L$ is uniquely determined by $\varpi_L$}. We define $(\rho)_{L}^{G}$ to be the irreducible representation of $G$ whose highest weight is $w\varpi_L$. In general, if $\rho=\oplus_i\rho_i$ is a finite-dimensional representation of $L$ with $\rho_i$ irreducible, we define
$$(\rho)_{L}^{G}=\oplus_i(\rho_i)_{L}^{G}.$$
\end{defn}
		
Now we are ready to state the conjecture.

\begin{conj}(Ben-Zvi--Sakellaridis--Venkatesh, Section 4.2.2 of \cite{BSV})\label{Whittaker induction}
With the notation above. If the dual of $\Delta_{red}$ is given by $\hat{\Delta}_{red}=(\hat{L}, \hat{H}_L',\hat{\iota}', \rho')$, then the dual of $\Delta$ is given by
$$(\hat{G},\hat{H}',\hat{\iota}',(\rho')_{\hat{H}_L'}^{\hat{H}'})$$
where $\hat{H}_L'$ is a Levi subgroup of $\hat{H}'$ and $\hat{H}'$ is generated by $\hat{H}_L'$ and $\{Im(\iota_\alpha)|\;\alpha\in \Delta_{\hat{G}}-\Delta_{\hat{M}}\}$. Here $\Delta_{\hat{G}}$ (resp. $\Delta_{\hat{M}}$) is the set of simple roots of $\hat{G}$ (resp. $\hat{L}$) and $\iota_\alpha:\SL_2\rightarrow \hat{G}$ is the embedding associated with $\alpha$.
\end{conj}

Assuming the above conjecture, it suffices to develop an algorithm for the duality in the reductive case, i.e. when $\Delta=(G,H,1,\rho_H)$. Within this setting, two particularly important special cases arise:
\begin{itemize}
	\item the case when $\rho_H=T(\tau)$ for some representation $\tau$ of $H$ (i.e. the spherical variety case/the polarized case);
	\item the case when $H=G$ (i.e. the symplectic vector space case).
\end{itemize}
In the first case, Ben-Zvi, Sakellaridis, and Venkatesh give an argument to compute the dual in \cite{BSV}. More specifically, in this case, if we use $\hat{\Delta}=(\hat{G},\hat{H}',\hat{\iota}',\rho_{\hat{H}'})$ to denote the dual quadruple, then $\hat{H}'$ and $\hat{\iota}'$ in the dual quadruple are given by the dual group of spherical varieties (defined in \cite{GN}, \cite{SV}, \cite{KS}), while $\rho_{\hat{H}'}$ is given by the theory of colors of spherical variety in \cite{S1}. The dual of many of the polarized cases has already been computed in \cite{KS} and \cite{S1} (see the tables at the end of those papers). In the second case (i.e. the vector space case), Mao, Wan, and Zhang have provided an algorithm to compute the dual quadruple in \cite{MWZ2}. Their results are summarized in the tables at the end of the paper.

Moreover, one can also define the nilpotent orbit $\hat{\iota}'$ in the dual quadruple for general case (which is compatible with the definition of the two special cases above) \footnote{although a general definition of the other two data in $\hat{\Delta}$ is still not available at this moment}. In fact, if we let $G'$ be the generic stabilizer of $G$ acting on $\CM_\Delta$, then there exists a Levi subgroup $L_\Delta$ of $G$ such that $[L_{\Delta},L_{\Delta}]\subset G'\subset L_\Delta$ and one can just let $\hat{\iota}'$  be the principal nilpotent orbit in $\hat{L}_{\Delta}$ where $\hat{L}_{\Delta}$ is the dual Levi subgroup of $L_{\Delta}$. In practice, the Levi subgroup $L_{\Delta}$ is easy to compute: for given $\Delta=(G,H,1,\rho_H)$, recall that we have let $H'$ be the generic stabilizer of $H$ acting on $\Fh^{\perp}$. From the structure theorem of spherical varieties, we know that there exists a Levi subgroup $L(X)$ of $G$ (here $X=G/H$) such that  $[L_{X},L_{X}]\subset H'\subset L_X$. Hence each Levi subgroup of $H'$ also corresponds to a Levi subgroup of $G$. Then $L_{\Delta}$ is given by the Levi subgroup 
of $H'$ associated with the multiplicity-free symplectic representation $\rho_H|_{H'}$ given in \cite{K} (which is also defined using the generic stabilizer).

Now we explain our strategy to compute the dual for $\Delta=(G,H,\iota,\rho_H)$. If $\Delta_{red}$ is polarized or a vector space, we can use the previous works of \cite{BSV} and \cite{MWZ2} to compute the dual of $\Delta_{red}$ and then use Conjecture \ref{Whittaker induction} to compute the dual of $\Delta$. This applies to most of the cases considered in this paper (including all the quadruples in Type $B_2=C_2$ and $G_2$ considered below).

\begin{rmk}
In a few cases that only appear in Type D (i.e. Type $\ast$ in Table \ref{Table D}), although $\Delta_{red}$ is not polarized or a vector space, we can still use some other arguments (via period integral conjecture and Conjecture \ref{Whittaker induction}) to compute its dual and hence the dual of $\Delta$. We refer the reader to the last part of Section 5 for details.
\end{rmk}

If $\Delta_{red}$ is not polarized or a vector space, there are two cases. In the first case, there exists a quadruple $\Delta'$ of $\hat{G}$ such that $(\Delta')_{red}$ is polarized or a vector space (in particular we can use the algorithm above to compute the dual of $\Delta'$) and the dual of $\Delta'$ is equal to $\Delta$. Then we know that the dual of $\Delta$ is just $\Delta'$. One example of such would be Quadruple 20 in Table \ref{Table A_{2n}}.

The second case occurs when no quadruple $\Delta'$ satisfies the conditions mentioned above. In such instances, we lack an algorithm to determine the dual quadruple.  Fortunately, when $G$ is simple, only one quadruple falls into this category: $\Delta=(\GL_8,\GL_6\times \GL_2,1,\wedge_{\GL_6}^3)$. Given that this is the sole remaining case, we conjecture it to be self-dual, consistent with the spirit of the duality. 

\begin{rmk}
It should be noted that when $G$ is not simple, additional examples of this type arise. In particular, the current algorithm—relying on the polarized case, the vector space case, and Conjecture \ref{Whittaker induction} for Whittaker induction—remains insufficient for computing the duality in full generality.
\end{rmk}

To end this section, we will use Types $B_2$ and $G_2$ as examples for our algorithm. We start with the $B_2$ case. As classified in the previous subsection (Table \ref{Table B_2 1}), there are 11 quadruples. All the 8 reductive cases (Quadruple 1-8 in the table) are polarized (Cases 1-3 are also vector space cases) and one can easily compute the dual quadruple. More specifically,  Quadruple $1$ is dual to $9$, $2$ is dual to $10$, $3$ is dual to $7$, $4$ is dual to $6$, $5$ is dual to $11$ and $8$ is self dual. This also gives us the dual for the remaining 3 non-reductive quadruples in the table as they are dual to $1,2,5$ respectively. Although we can also use Conjecture \ref{Whittaker induction} to compute the dual for these three quadruples (since $\Delta_{red}$ is polarized in these cases). For example, for Quadruple $10$, $\Delta_{red}=(\GL_2,\GL_1,1,0)$ whose dual is $(\GL_2,\GL_2,1,T(std))$. Then by Conjecture \ref{Whittaker induction} we know that the dual of $\Delta$ is Quadruple $2$ in the table. The arguments for the remaining two cases are similar. In this case, the result is summarized in Table \ref{Table B_2 2}.

Next, we consider the $G_2$ case. There are 4 quadruples as listed in Table \ref{Table G_2 1}. The first two are reductive and polarized, and one can easily compute the dual quadruple. We will have $1$ dual to $4$ and $2$ dual to $3$. This also gives us the dual for the remaining 2 non-reductive quadruples in the table as they are dual to $1,2$ respectively. As in the previous case, we can also use Conjecture \ref{Whittaker induction} to compute the dual for these two quadruples since $\Delta_{red}$ is polarized in these cases. In this case, the result is summarized in Table \ref{Table G_2 2}.

\begin{rmk}
As in \cite{MWZ2}, for some cases considered in this paper, it is also possible to provide evidence for the duality via the period integral conjecture. However, we do not explore this direction in the present paper.
\end{rmk}

\section{Type $A$}\label{sec:type-A}
\subsection{Classification}
For simplicity, we define a nilpotent orbit $\iota$ to be \emph{Levi-spherical} when $G_\iota$ is a spherical subgroup of $L$.

We determine the hyperspherical data $(G,H,\iota,\rho_H)$ by examining the following four criteria as follows:
\begin{enumerate}
    \item \label{item:Condition-A1} The nilpotent orbit $\iota$ is Levi-spherical.
    \item \label{item:Condition-A2} $\rho_\iota=\Fu/\Fu^+$ is a coisotropic representation of $G_\iota$ and its generic stabilizer $G_\iota'$.
    \item \label{item:Condition-A3} $H \subseteq G_\iota$ is a spherical subgroup of $L$ and its generic stabilizer $H'$ in $L$ is connected.
    \item \label{item:Condition-A4} $\rho_H\oplus \Fu/\Fu^+$ is an anomaly-free coisotropic representation of $H$ and $H'$.
\end{enumerate}

\subsubsection{Determining $\iota$}
We first determine the Levi-spherical nilpotent orbits of $\GL_n$ that are even (i.e., all parts of the partition have the same parity).

\begin{lemma}
    Let $G$ be the general linear group $\GL_n$. If $\iota$ is a Levi-spherical even nilpotent orbit of $G$, then $\iota$ corresponds to one of the following partitions:
    \begin{itemize}
        \item $(n,1^{a_1})$ (hook-type)
        \item $(2^{a_2})$ (Shalika)
        \item The `exceptional' partition $(3^2)$.
    \end{itemize}

\end{lemma}

\begin{proof}
We write the partition associated with $\iota$ as $(p_1^{a_1},p_2^{a_2},\dots, p_{\ell}^{a_\ell})$ with $p_1>p_2>\cdots >p_\ell>0$. It is easy to see that: \[G_\iota\simeq \prod_{i=1}^\ell \GL_{a_i},\quad L\simeq\prod_{i=1}^\ell \GL_{a_1+a_2+\cdots+a_{i}}^{\times (p_{i}-p_{i+1})}\] with $p_{\ell+1}=0$. (Refer to Theorem 6.1.3 in \cite{CM93}, for instance.) 
Furthermore, the embedding $G_\iota\hookrightarrow L$ gives a non-trivial composite homomorphism \[\GL_{a_i}\to G_\iota \hookrightarrow L \twoheadrightarrow  \GL_{a_1+a_2+\cdots+a_{j}}^{\times (p_{j}-p_{j+1})}\] whenever $i\leq j$.

Let us consider the embedding of $G_{\iota}$ into the factor $\GL_{a_1+a_2+\cdots+a_{\ell}}^{\times p_{\ell}}$ of $L$. 
Since $G_\iota$ is a reductive spherical subgroup of $L$, it follows that $\prod_{i=1}^l \GL_{a_i}$ is a spherical subgroup of $\GL_{a_1+a_2+\cdots+a_{\ell}}^{\times p_{\ell}}$.
Due to the classification of \cite{KS}, $\prod_{i=1}^l \GL_{a_i}$ is a spherical subgroup if and only if $\ell = 1$ or $\ell = 2$ with $p_2 = 1$.

If $\ell = 2$ and $p_2 = 1$, then $G_\iota = \GL_{a_1}\times\GL_{a_2}$ is embedded into $L = \GL_{a_1+a_2}\times \GL_{a_1}^{\times (p_1 - 1)}$.
For $(L,G_\iota)$ to be a spherical pair, the restriction of the factor $\GL_{a_1}^{\times (p_1 - 1)}$ in $G_\iota$ also has to be a spherical subgroup of $\GL_{a_1}^{\times (p_1 - 1)}$. 
By the classification of \cite{KS}, the possible cases are the following:

\begin{itemize}
    \item $a_1 = 1$. In this case, $\iota = (p_1,1^{a_2})$ is the hook-type orbit.
    \item $a_1 = 2$ and $p_1 = 4$. This is not possible since $\iota$ is even, and $p_1, p_2$ have the same parity.
    \item $a_1 \geq 2$ and $p_1 = 3$. By the classification of \cite{KS}, there are no spherical subgroups of this form.
\end{itemize}

If $\ell = 1$, then $G_\iota = \GL_{a_1}$ and $L = \GL_{a_1}^{\times p_1}$. 
By the classification of \cite{KS}, we have the following cases:
\begin{itemize}
    \item $p_1 = 1$. In this case, $\iota = (1^{a_1})$ is the trivial orbit.
    \item $p_1 = 2$. In this case, $\iota = (2^{a_1})$ is the Shalika orbit.
    \item $p_1 = 3$ and $a_1 = 2$. In this case, $\iota = (3^2)$ is the exceptional orbit.
\end{itemize}

In summary, the only Levi-spherical even nilpotent orbits of $\GL_n$ are the ones listed in the lemma.
\end{proof}

Now let $\iota$ be a non-even nilpotent orbit of $\GL_n$ and let $\iota_{odd}$ (resp. $\iota_{even}$) be the odd (resp. even) part of its corresponding partition. 
Similar to the argument of Section 5.4 in \cite{BP}, $L$ and $G_\iota$ can be decomposed along $\iota_{odd}$ and $\iota_{even}$:
\[L = L_{{odd}} \times L_{{even}},\quad  G_\iota = G_{\iota, {odd}} \times G_{\iota, {even}},\]
where $L_{{odd}}$ (resp. $L_{{even}}$) and $G_{\iota, {odd}}$ (resp. $G_{\iota, {even}}$) are the Levi subgroup and stabilizer corresponding to the partition $\iota_{odd}$ (resp. $\iota_{even}$).
Therefore, $\iota$ is Levi-spherical if and only if both $\iota_{odd}$ and $\iota_{even}$ are Levi-spherical. That is:
\begin{itemize}
    \item $\iota_{odd}$ is one of the following: \((1^k),(2m+1,1^k),(3^2)\).
    \item $\iota_{even}$ is one of the following: \((2^n),(2n)\).
\end{itemize}

We verify Condition \eqref{item:Condition-A2} for each combination; that is, we check whether $\Fu/\Fu^+$ is a coisotropic representation of $G_\iota$ and $G_\iota'$:

\begin{itemize}
    \item If $\iota_{odd}=(1^k)$, then $L_{odd} = G_{\iota,{odd}}=G_{\iota,odd}' = \GL_k$.
    \begin{itemize}
        \item If $\iota_{even}=(2^n)$, then $L=\GL_n\times \GL_n\times \GL_k$, $G_\iota = \GL_k\times \GL_n$, $\Fu/\Fu^+ = T(std_{\GL_k} \otimes std_{\GL_n})$ and $G_{\iota}' = \GL_k\times \BG_m^n $. Referring to the tables in \cite{K}, $\Fu/\Fu^+$ is a coisotropic representation of $G_{\iota}'$ if and only if $k=1$ or $n\leq 2$.
        \item If $\iota_{even}=(2n)$, then $L=\GL_k\times \GL_{1}^{2n}$, $G_\iota=G_{\iota}'=\GL_k\times \GL_1$, and $\Fu/\Fu^+ = T(std_{\GL_k}\otimes std_{\GL_1})$. Referring to the tables in \cite{K}, $\Fu/\Fu^+$  is coisotropic in this case.
    \end{itemize}
    \item If $\iota_{odd}=(2m+1,1^k)$, then $L_{odd} = \GL_{k+1}\times \GL_1^{2m}$, $G_{\iota,odd} = \GL_{k}\times \GL_1$.
    \begin{itemize}
        \item If $\iota_{even}=(2^n)$, then $L=\GL_n\times \GL_n\times \GL_{k+1}\times \GL_1^{2m}$, $G_\iota=\GL_n\times \GL_{k}\times \GL_1$,   $\Fu/\Fu^+ = T(std_{\GL_k} \otimes std_{\GL_n})\oplus T(std_{\GL_1} \otimes std_{\GL_n})^{\oplus 2}$, and $G_\iota'=\BG_m^n\times \GL_{k-1}\times \GL_1$. Referring to the tables in \cite{K}, $\Fu/\Fu^+$ is not a coisotropic representation of $G_\iota'$.
        \item If $\iota_{even}=(2n)$, then $L=\GL_{1}^{2n}\times \GL_{k+1}\times \GL_1^{2m}$, $G_\iota=\GL_1\times \GL_k\times \GL_1$, 
        $\Fu/\Fu^+ = T(std_{\GL_{k}}\otimes std_{\GL_{1}}) \oplus T(std_{\GL_{1}}\otimes std_{\GL_{1}})^{\oplus\min\{2n-1,2m\}}$, and $G_\iota=\GL_1\times \GL_{k-1}\times \GL_1$. Referring to the tables in \cite{K}, $\Fu/\Fu^+$ is not a coisotropic representation of $G_\iota'$.
    \end{itemize}
    \item If $\iota_{odd}=(3^2)$, then $L_{odd} = \GL_2^{3}$ and $G_{\iota,odd} = \GL_2$.
    \begin{itemize}
        \item If $\iota_{even}=(2^n)$, then $L=\GL_{n}^2\times \GL_{2}^{3}$, $G_\iota=\GL_2\times \GL_n$, $\Fu/\Fu^+ = T(std_{\GL_2} \otimes std_{\GL_n})^{\oplus 2}$ and $G_\iota'=\GL_1\times \BG_m^n$. Referring to the tables in \cite{K}, $\Fu/\Fu^+$ is not a coisotropic representation of $G_\iota'$.
        \item If $\iota_{even}=(2n)$, then $L=\GL_{1}^{2n}\times \GL_{2}^3$, $G_\iota=\GL_1\times \GL_2$, $\Fu/\Fu^+ = T(std_{\GL_2}\otimes std_{\GL_1})^{\oplus i}$ where $i=2$ if $n=1$ and $i=3$ if $n>1$, and $G_\iota'=\GL_1\times \GL_1$. Referring to the tables in \cite{K}, $\Fu/\Fu^+$ is not a coisotropic representation of $G_\iota'$.
    \end{itemize}
\end{itemize}

To summarize, the only non-even nilpotent orbits satisfying Conditions \eqref{item:Condition-A1} and \eqref{item:Condition-A2} are $(2m,1^n),(2^n,1),(2^2,1^n)$. As a result, we know that when $G$ is of Type $A$, the set $Nil_0(\Fg)$ is equal to 
$$\{(n,1^m),(2^n),(3^2),(2^n,1),(2^2,1^n)\}.$$

\subsubsection{Determining $H$ and $\rho_H$}

We now list the anomaly-free coisotropic representations of certain groups for reference \footnote{To simplify the notation, we do not distinguish between groups of the same root types throughout this paper. For the Type $A$ cases considered in this section, it is occasionally necessary to replace the general linear group with the special linear group to ensure a representation is symplectic  (e.g. the exterior cube representation of $\GL_6$). But to simplify the notation we will still write the general linear group.}. 

\begin{itemize}
    \item $\GL_n$: $0$, $T(std)$, $T(std\oplus std)$, $T(\wedge^2)$, $T(\wedge^2 \oplus std)$ for general $n$; $\wedge^3,\wedge^3\oplus std$ when $n=6$.
    \item $\GL_n\times \GL_m$ with $m,n>1$: $T(std_{\GL_m}\otimes std_{\GL_n})$, $T(std_{\GL_m}\otimes std_{\GL_n}\oplus std_{\GL_n})$ for general $n,m$; $T(\wedge_{\GL_4}^{2}\otimes std_{\GL_2}), \;T(\wedge_{\GL_4}^{2}\otimes std_{\GL_2}\oplus std_{\GL_4}),\;T(\wedge_{\GL_4}^{2}\otimes std_{\GL_2}\oplus std_{\GL_4}\oplus std_{\GL_@}),\;T(\wedge_{\GL_4}^{2}\otimes std_{\GL_2}\oplus std_{\GL_4}\otimes std_{\GL_2})$ when $n=4$ and $m=2$; $T(\wedge_{\GL_6}^{3}\otimes std_{\GL_2}\oplus std_{\GL_6})$ when $n=6$ and $m=2$.
    \item $\Sp_{2n}$: $0$, $T(std)$, $T(std\oplus std)$ for general $n$, $T(std_{\SO_5})$ when $n=2$, $T(\wedge_{0}^{3}\oplus std)$ when $n=3$. 
\end{itemize}
Note that the $T(Sym^2)$ representation of $\GL_n$ and the $Sym^3$ representation of $\GL_2$ are excluded since the generic stabilizer is not connected.

We will analyze each nilpotent orbit in $Nil_0(\Fg)$ one by one to find all the $H$ and $\rho_H$ that satisfy Conditions \eqref{item:Condition-A3} and \eqref{item:Condition-A4}.

\begin{itemize}
    \item When $\iota$ is the trivial orbit: $\iota=(1^n)$, $L=G_\iota=\GL_n$ and $\Fu/\Fu^+=0$. In this case $H$ can be all possible reductive spherical subgroups of $\GL_n$.
    \begin{itemize}
        \item Let $H=G_\iota$. Then $\rho_H$ can be any anomaly-free coisotropic representation of $\GL_n$, which are $0$, $T(std)$, $T(std\oplus std)$, $T(\wedge^2)$, $T(\wedge^2 \oplus std)$ for general $n$, and $\wedge^3,\wedge^3\oplus std$ when $n=6$.
    
        \item Let $H=\GL_m\times \GL_{n-m}$. The generic stabilizer is $H'=\GL_{n-2m}\times \GL_1^{m}$. In this case, $\rho_H$ can be $0$, $T(\GL_m)$ or $T(\GL_{n-m})$ for general $n,m$. If $m=1$, $\rho_H$ can also be $T(\wedge^2_{\GL_{n-1}})$ for general $n$. If $m=1$ (resp. $m=2$) and $n-m=6$, $\rho_H$ can also be $\wedge_{\GL_6}^3$.

        \item Let $n=2m+1$ and $H=\Sp_{2m}$. The generic stabilizer is trivial. The only possible choice for $\rho_H$ is $0$.

        \item Let $n=2m$ and $H=\Sp_{2m}$ with $m>1$. The generic stabilizer is $H'=\SL_2^{m}$. In this case, $\rho_H$ can be $0$ or $T(std)$.
    \end{itemize}

    \item When $\iota = (n)$ is the regular orbit, we have $L=\GL_{1}^{n}$ and $G_\iota=\GL_{1}$. $\Fu/\Fu^+=0$. In this case we must have $H=G_\iota$ and $\rho_H=0$.

    \item When $\iota = (2m+1,1^{n})$ is an even hook-type orbit, $L=\GL_{n+1}\times \GL_1^{2m}$, $G_\iota = \GL_{n}\times \GL_{1} $ and $\Fu/\Fu^+=0$. In this case, when $n$ is odd, the only possible $H$ is $G_\iota$; when $n$ is even, $H$ can be $G_\iota$ or $\Sp_n$. When $H=G_\iota$, the generic stabilizer is $H'=\GL_{n-1}\times \GL_{1}$ and $\rho_H$ can be $0$, $T(std)$, or $T(\wedge^2)$. When $n$ is even and $H=\Sp_n$, the generic stabilizer is trivial and $\rho_H$ has to be $0$.

    \item When $\iota = (2^{n})$ is the ``Shalika" orbit with $n>1$, $L=\GL_{n}\times \GL_{n} $, $G_\iota=\GL_{n}$ and $\Fu/\Fu^+=0$. In this case, the only possible $H$ is also $G_\iota$. The generic stabilizer is $H'=\BG_{m}^{n} $ and $\rho_H$ can be $0$ or $T(std)$.

    \item When $\iota = (2m,1^{n})$ is a non-even hook-type orbit, $L=\GL_{n}\times \GL_{1}^{ 2m}$, $G_\iota=\GL_{n}\times \GL_{1} $ and $\Fu/\Fu^+=T(std_{\GL_n}\otimes std_{\GL_1})$. The argument is similar to the trivial orbit case, except that $T(std_{\GL_n}\otimes std_{\GL_1})$ should be a subrepresentation of $\rho_H\oplus \Fu/\Fu^+$. Thus, the possible cases are $H=G_\iota$ or $H=\Sp_{n}$ (when $n$ is even). 
    \begin{itemize}
        \item If $H=G_\iota$, $\rho_H$ can be $0,T(std_{\GL_n}),T(\wedge_{\GL_n}^{2})$.
        \item If $n$ is even and $H=\Sp_{n}$, $\rho_H$ has to be $0$.
    \end{itemize}

    \item When $\iota = (2^{n},1)$, $L=\GL_{n}^{2}\times \GL_{1} $, $G_\iota=\GL_{n}\times \GL_{1} $ and $\Fu/\Fu^+=T(std_{\GL_n}\otimes std_{\GL_1})$. The only possible case is $H=G_\iota$ and $\rho_H=0$.

    \item When $\iota = (2^{2},1^{n})$, $L=\GL_{2}^{2}\times \GL_{n} $, $G_\iota=\GL_{2}\times \GL_{n} $ and $\Fu/\Fu^+=T(std_{\GL_2}\otimes std_{\GL_n})$. In this case we must have $H=G_\iota$ and $\rho_H=0$.

    \item When $\iota = (3^2)$, $L=\GL_2^{3}$, $G_\iota=\GL_2$ and $\Fu/\Fu^+=0$. In this case $H$ has to be $G_\iota$. The generic stabilizer is $\GL_1$ and $\rho_H$ has to be $0$.
\end{itemize}

To summarize, the tables below present the anomaly-free hyperspherical BZSV quadruples for Type $A$. We recall that $H'$ is the generic stabilizer of $H$ acting on $\Fh^{\perp}$ and $\rho_{H,\iota}=\Fu/\Fu^+\oplus \rho_H$. We separate the cases by parity. In these tables, $m$ is an integer such that $m\leq n$. Also we will skip the $\GL_1$-action in some representations (i.e. we may just write the representation $std_{\GL_n}\otimes std_{\GL_1}$ as $std_{\GL_n}$).

\begin{landscape}
\thispagestyle{empty}
\centering
\resizebox{\linewidth}{!}{%
\begin{tabular}{|c|c|c|c|c|c|}
\hline
\textnumero & $\Delta=(G,H,\iota,\rho_H)$  & $[L,L]$  & $\Fu/\Fu^{+}$  & $H'$  & $\rho_{H,\iota}|_{H'}$ \\
\hline
1  & $(\GL_{2n+1},\GL_{2n+1},(1^{2n+1}),0)$  & $\GL_{2n+1}$  & $0$  & $\GL_{2n+1}$  & $0$ \\
\hline 
2  & $(\GL_{2n+1},\GL_{2n+1},(1^{2n+1}),T(std_{\GL_{2n+1}}))$  & $\GL_{2n+1}$  & $0$  & $\GL_{2n+1}$  & $T(std_{\GL_{2n+1}})$ \\
\hline 
3  & $(\GL_{2n+1},\GL_{2n+1},(1^{2n+1}),T(\wedge^{2}))$  & $\GL_{2n+1}$  & $0$  & $\GL_{2n+1}$  & $T(\wedge^{2})$ \\
\hline 
4  & $(\GL_{2n+1},\GL_{2n+1},(1^{2n+1}),T(std_{\GL_{2n+1}}\oplus std_{\GL_{2n+1}}))$  & $\GL_{2n+1}$  & $0$  & $\GL_{2n+1}$  & $T(std_{\GL_{2n+1}}\oplus std_{\GL_{2n+1}})$ \\
\hline 
5  & $(\GL_{2n+1},\GL_{2n+1},(1^{2n+1}),T(\wedge_{}^{2}\oplus std_{\GL_{2n+1}}))$  & $\GL_{2n+1}$  & $0$  & $\GL_{2n+1}$  & $T(\wedge^{2}\oplus std_{\GL_{2n+1}})$ \\
\hline 
6  & $(\GL_{2n+1},\GL_{2n-m+1}\times\GL_{m},(1^{2n+1}),0)$  & $\GL_{2n+1}$  & $0$  & $\GL_{2n-2m+1}\times \GL_1^{m}$  & $0$ \\
\hline 
7  & $(\GL_{2n+1},\GL_{2n-m+1}\times\GL_{m},(1^{2n+1}),T(std_{\GL_{2n-m+1}}))$  & $\GL_{2n+1}$  & $0$  & $\GL_{2n-2m+1}\times \GL_1^{m}$  & $T(std_{\GL_{2n-2m+1}})\oplus T(std_{\GL_{1}}^{\oplus m})$ \\
\hline 
8  & $(\GL_{2n+1},\GL_{2n-m+1}\times\GL_{m},(1^{2n+1}),T(std_{\GL_{m}}))$  & $\GL_{2n+1}$  & $0$  & $\GL_{2n-2m+1}\times \GL_1^{m}$  & $T(std_{\GL_{1}}^{\oplus m})$ \\
\hline 
9  & $(\GL_{2n+1},\GL_{2n},(1^{2n+1}),T(\wedge^{2}))$  & $\GL_{2n+1}$  & $0$  & $\GL_{2n-1}\times \GL_1^2$  & $T(\wedge^{2}\oplus std_{\GL_{2n-1}})$ \\
\hline 
10  & $(\GL_{2n+1},\Sp_{2n},(1^{2n+1}),0)$  & $\GL_{2n+1}$  & $0$  & $1$  & $0$ \\
\hline 
11  & $(\GL_{2n+1},\GL_{2}\times\GL_{2n-3},(2^{2},1^{2n-3}),0)$  & $\GL_{2}^{2}\times\GL_{2n-3}$  & $T(std_{\GL_{2}}\otimes std_{\GL_{2n-3}})$  & $\GL_1^{ 2}\oplus\GL_{2n-3}$  & $T(std_{\GL_{2n-3}})^{\oplus2}$ \\
\hline 
12  & $(\GL_{2n+1},\GL_{n},(2^{n},1),0)$  & $\GL_{n}^{ 2}$  & $T(std)$  & $\GL_1^{ n}$  & $T(std_{\GL_{1}}^{\oplus n})$ \\
\hline 
13  & $(\GL_{2n+1},\GL_{2m+1},(2n-2m,1^{2m+1}),0)$  & $\GL_{2m+1}$  & $T(std_{\GL_{2m+1}})$  & $\GL_{2m+1}$  & $T(std_{\GL_{2m+1}})$ \\
\hline 
14  & $(\GL_{2n+1},\GL_{2m+1},(2n-2m,1^{2m+1}),T(std_{\GL_{2m+1}}))$  & $\GL_{2m+1}$  & $T(std_{\GL_{2m+1}})$  & $\GL_{2m+1}$  & $T(std_{\GL_{2m+1}}\oplus std_{\GL_{2m+1}})$ \\
\hline 
15  & $(\GL_{2n+1},\GL_{2m+1},(2n-2m,1^{2m+1}),T(\wedge^{2}))$  & $\GL_{2m+1}$  & $T(std_{\GL_{2m+1}})$  & $\GL_{2m+1}$  & $T(\wedge^{2}\oplus std_{\GL_{2m+1}})$ \\
\hline 
16  & $(\GL_{2n+1},\GL_{2m},(2n-2m+1,1^{2m}),0)$  & $\GL_{2m+1}$  & $0$  & $\GL_{2m-1}\times \GL_1$  & $0$ \\
\hline 
17  & $(\GL_{2n+1},\GL_{2m},(2n-2m+1,1^{2m}),T(std_{\GL_{2m}}))$  & $\GL_{2m+1}$  & $0$  & $\GL_{2m-1}\times \GL_1$  & $T(std_{\GL_{2m-1}}\oplus std_{\GL_{1}})$ \\
\hline 
18  & $(\GL_{2n+1},\GL_{2m},(2n-2m+1,1^{2m}),T(\wedge^{2}))$  & $\GL_{2m+1}$  & $0$  & $\GL_{2m-1}\times \GL_1$  & $T(\wedge^{2}\oplus std_{\GL_{2m-1}})$ \\
\hline 
19  & $(\GL_{2n+1},\Sp_{2m},(2n-2m+1,1^{2m}),0)$  & $\GL_{2m+1}$  & $0$  & $1$  & $0$ \\
\hline 
20  & $(\GL_{2n+1},\GL_{6},(2n-5,1^{6}),\wedge^{3})$  & $\GL_{7}$  & $0$  & $\GL_{5}\times \GL_1$  & $T(\wedge^{2})$ \\
\hline 
21  & $(\GL_{2n+1},1,(2n+1),0)$  & $1$  & $0$  & $1$  & $0$ \\
\hline 
\end{tabular}
}
\captionof{table}{Type $A_{2n}$}
\label{Table A_{2n}}
\vfill
\begin{center}
\thepage
\end{center}
\end{landscape}

\begin{landscape}
\thispagestyle{empty}
\centering
\resizebox{\linewidth}{!}{%
\begin{tabular}{|c|c|c|c|c|c|}
\hline
\textnumero & $\Delta=(G,H,\iota,\rho_H)$  & $[L,L]$  & $\Fu/\Fu^{+}$  & $H'$  & $\rho_{H,\iota}|_{H'}$ \\
\hline
1  & $(\GL_{2n},\GL_{2n},(1^{2n}),0)$  & $\GL_{2n}$  & $0$  & $\GL_{2n}$  & $0$ \\
\hline 
2  & $(\GL_{2n},\GL_{2n},(1^{2n}),T(std_{\GL_{2n}}))$  & $\GL_{2n}$  & $0$  & $\GL_{2n}$  & $T(std_{\GL_{2n}})$ \\
\hline 
3  & $(\GL_{2n},\GL_{2n},(1^{2n}),T(\wedge^{2}))$  & $\GL_{2n}$  & $0$  & $\GL_{2n}$  & $T(\wedge^{2})$ \\
\hline 
4  & $(\GL_{2n},\GL_{2n},(1^{2n}),T(std_{\GL_{2n}}\oplus std_{\GL_{2n}}))$  & $\GL_{2n}$  & $0$  & $\GL_{2n}$  & $T(std_{\GL_{2n}}\oplus std_{\GL_{2n}})$ \\
\hline 
5  & $(\GL_{2n},\GL_{2n},(1^{2n}),T(\wedge^{2}\oplus std_{\GL_{2n}}))$  & $\GL_{2n}$  & $0$  & $\GL_{2n}$  & $T(\wedge^{2}\oplus std_{\GL_{2n}})$ \\
\hline 
6  & $(\GL_{2n},\GL_{2n-m}\times\GL_{m},(1^{2n}),0)$  & $\GL_{2n}$  & $0$  & $\GL_{2n-2m}\times \GL_1^{ m}$  & $0$ \\
\hline 
7  & $(\GL_{2n},\GL_{2n-m}\times\GL_{m},(1^{2n}),T(std_{\GL_{2n-m}}))$  & $\GL_{2n}$  & $0$  & $\GL_{2n-2m}\times \GL_1^{ m}$  & $T(std_{\GL_{2n-2m}})\oplus T(std_{\GL_{1}}^{\oplus m})$ \\
\hline 
8  & $(\GL_{2n},\GL_{2n-m}\times\GL_{m},(1^{2n}),T(std_{\GL_{m}}))$  & $\GL_{2n}$  & $0$  & $\GL_{2n-2m}\times \GL_1^{ m}$  & $T(std_{\GL_{1}}^{\oplus m})$ \\
\hline 
9  & $(\GL_{2n},\GL_{2n-1},(1^{2n}),T(\wedge^{2}))$  & $\GL_{2n}$  & $0$  & $\GL_{2n-2}\times\GL_1^{ 2}$  & $T(\wedge^{2}\oplus std_{\GL_{2n-2}})$ \\
\hline 
10  & $(\GL_{2n},\Sp_{2n},(1^{2n}),0)$  & $\GL_{2n}$  & $0$  & $\SL_2^{ n}$  & $0$ \\
\hline 
11  & $(\GL_{2n},\Sp_{2n},(1^{2n}),T(std_{\Sp_{2n}}))$  & $\GL_{2n}$  & $0$  & $\SL_2^{ n}$  & $T(std_{\SL_{2}}^{\oplus n})$ \\
\hline 
12  & $(\GL_{2n},\GL_{2}\times\GL_{2n-4},(2^{2},1^{2n-4}),0)$  & $\GL_{2}^{ 2}\times\GL_{2n-4}$  & $T(std_{\GL_{2}}\otimes std_{\GL_{2n-4}})$  & $\GL_1^{ 2}\oplus\GL_{2n-4}$  & $T(std_{\GL_{2n-4}})^{\oplus2}$ \\
\hline 
13  & $(\GL_{2n},\GL_{n},(2^{n}),0)$  & $\GL_{n}^{ 2}$  & $0$  & $\GL_1^{ n}$  & $0$ \\
\hline 
14  & $(\GL_{2n},\GL_{n},(2^{n}),T(std_{\GL_{n}}))$  & $\GL_{n}^{ 2}$  & $0$  & $\GL_1^{ n}$  & $T(std_{\GL_{1}}^{\oplus n})$ \\
\hline 
15  & $(\GL_{2n},\GL_{2m+1},(2n-2m-1,1^{2m+1}),0)$  & $\GL_{2m+2}$  & $0$  & $\GL_{2m}\times \GL_1$  & $0$ \\
\hline 
16  & $(\GL_{2n},\GL_{2m+1},(2n-2m-1,1^{2m+1}),T(std_{\GL_{2m+1}}))$  & $\GL_{2m+2}$  & $0$  & $\GL_{2m}\times \GL_1$  & $T(std_{\GL_{2m}}\oplus std_{\GL_{1}})$ \\
\hline 
17  & $(\GL_{2n},\GL_{2m+1},(2n-2m-1,1^{2m+1}),T(\wedge^{2}))$  & $\GL_{2m+2}$  & $0$  & $\GL_{2m}\times \GL_1$  & $T(\wedge^{2}\oplus std_{\GL_{2m}})$ \\
\hline 
18  & $(\GL_{2n},\GL_{2m},(2n-2m,1^{2m}),0)$  & $\GL_{2m}$  & $T(std_{\GL_{2m}})$  & $\GL_{2m}$  & $T(std_{\GL_{2m}})$ \\
\hline 
19  & $(\GL_{2n},\GL_{2m},(2n-2m,1^{2m}),T(std_{\GL_{2m}}))$  & $\GL_{2m}$  & $T(std_{\GL_{2m}})$  & $\GL_{2m}$  & $T(std_{\GL_{2m}}\oplus std_{\GL_{2m}})$ \\
\hline 
20  & $(\GL_{2n},\GL_{2m},(2n-2m,1^{2m}),T(\wedge^{2}))$  & $\GL_{2m}$  & $T(std_{\GL_{2m}})$  & $\GL_{2m}$  & $T(\wedge^{2}\oplus std_{\GL_{2m}})$ \\
\hline 
21  & $(\GL_{2n},\Sp_{2m},(2n-2m,1^{2m}),0)$  & $\GL_{2m}$  & $T(std_{\Sp_{2m}})$  & $\SL_{2}^{ m}$  & $T(std_{\SL_2}^{\oplus m})$ \\
\hline 
22  & $(\GL_{2n},\GL_{6},(2n-6,1^{6}),\wedge^{3})$  & $\GL_{6}$  & $T(std_{\GL_{6}})$  & $\GL_{6}$  & $\wedge^{3}\oplus T(std_{\GL_{6}})$ \\
\hline 
23  & $(\GL_{2n},1,(2n),0)$  & $1$  & $0$  & $1$  & $0$ \\
\hline 
24  & $(\GL_{6},\GL_{6},(1^{6}),\wedge^{3})$  & $\GL_{6}$  & $0$  & $\GL_{6}$  & $\wedge^{3}$ \\
\hline 
25  & $(\GL_{6},\GL_{6},(1^{6}),\wedge^{3}\oplus T(std_{\GL_{6}}))$  & $\GL_{6}$  & $0$  & $\GL_{6}$  & $\wedge^{3}\oplus T(std_{\GL_{6}})$ \\
\hline 
26  & $(\GL_{6},\GL_{2},(3^{2}),0)$  & $\GL_{2}^{ 3}$  & $0$  & $\GL_1$  & $0$ \\
\hline 
27  & $(\GL_8, \GL_6\times \GL_2, (1^8), \wedge^3_{\GL_6})$ & $\GL_8$ & $0$ &  $\GL_4\times \GL_2$ & $T(std_{\GL_{4}})\oplus\wedge^{2}_{\GL_4}\otimes std_{\GL_2}$\\ 
\hline 
\end{tabular}
}
\captionof{table}{Type $A_{2n-1}$}
\label{Table A_{2n-1}}
\vfill
\begin{center}
\thepage
\end{center}
\end{landscape}

\subsection{The Dual}
As explained in the previous chapter, we use the following three strategies to compute the dual.
\begin{enumerate}
\item The polarized cases studied in \cite{BSV}.
\item The vector space cases studied in \cite{MWZ2}.
\item Conjecture \ref{Whittaker induction} regarding the behavior of the dual under Whittaker induction.
\end{enumerate}

We now examine the dual of the quadruples in the above tables case by case. We begin with the cases where $G=\GL_{2n+1}$.

\begin{itemize}
    \item Quadruples 1-5 are vector space cases. We refer to \cite{MWZ2} for their duals:
    
    \begin{itemize}
    \item Quadruple 1 is dual to $(\GL_{2n+1},1,(2n+1),0)$, which is Quadruple 21.
    \item Quadruple 2 is dual to $(\GL_{2n+1},\GL_1,(2n,1),0)$, which is Quadruple 13 with $m=0$.
    \item Quadruple 3 is dual to $(\GL_{2n+1},\GL_{n},(2^{n},1),0)$, which is Quadruple 12.
    \item Quadruple 4 is dual to $(\GL_{2n+1},\GL_2,(2n-1,1^2),T(std_{\GL_2}))$, which is Quadruple 17 with $m=1$.
    \item Quadruple 5 is dual to $(\GL_{2n+1},\GL_{n+1}\times\GL_n,1,T(std_{\GL_{n+1}}))$, which is Quadruple 7 with $m=n$.
    \end{itemize}
    
    \item For Quadruples 6-10, $\iota$ is the trivial orbit, and they are all polarized cases. We can use the algorithm in \cite{BSV} to compute the dual. 
    \begin{itemize}
        \item Quadruple 6 is dual to $(\GL_{2n+1}, \Sp_{2m}, (2n-2m+1,1^{2m}), 0)$, which is Quadruple 19.
        \item Quadruple 7 is dual to $(\GL_{2n+1},\GL_{2n+1},1,T(std\oplus \wedge^2))$ when $m=n$ (this is Quadruple 5); and is dual to
        $(\GL_{2n+1}, \GL_{2m+1}, (2n-2m,1^{2m+1}), T(\wedge^2))$ when $m<n$ (this is Quadruple 15). 
        \item Quadruple 8 is dual to $(\GL_{2n+1}, \GL_{2m}, (2n-2m+1,1^{2m}), T(\wedge^2))$, which is   Quadruple 18 when $m<n$ and Quadruple 9 when $m=n$.
        \item Quadruple 9 is dual to $(\GL_{2n+1}, \GL_{n+1}\times \GL_{n}, (1^{2n+1}), T(std_{\GL_{n}}))$, which is Quadruple 8 with $m=n$.
        \item Quadruple 10 is dual to $(\GL_{2n+1}, \GL_{n+1}\times \GL_{n}, (1^{2n+1}), 0)$ which is Quadruple 6 when $m=n$.
    \end{itemize}
    \item For Quadruple 11, $L=\GL_2^{ 2}\times \GL_{2n-3}$, and  $\Delta_{red}=(\GL_2^{ 2}\times \GL_{2n-3},\GL_{2}\times\GL_{2n-3},0,T(std_{\GL_{2}}\otimes std_{\GL_{2n-3}}))$ which is a polarized case. Its dual is $\hat\Delta_{red}=(\GL_2^{ 2}\times \GL_{2n-3},\GL_{2}^{ 2}\times \GL_{2},1^2\times(2n-5,1^2), std_{\GL_{2}}\otimes std_{\GL_{2}} \otimes std_{\GL_{2}})$. By Conjecture \ref{Whittaker induction}, $\hat\Delta$ is $(\GL_{2n+1},\GL_{6},(2n-5,1^6),\wedge^3)$, which is Quadruple 20.

    \item For Quadruple 12, $\iota=(2^n,1)$, $L=\GL_n\times \GL_n$, and $\Delta_{red}=(\GL_n\times\GL_n,\GL_n,1,T(std_{\GL_n}))$. This is a polarized case, and the dual is $\hat\Delta_{red}=(\GL_n\times\GL_n,\GL_n\times \GL_n,1,T(std_{\GL_n}\otimes std_{\GL_n}))$.
    By Conjecture \ref{Whittaker induction}, $\hat\Delta = (\GL_{2n+1},\GL_{2n+1},1,T(\wedge^2))$, which is  Quadruple 3.

    \item For Quadruples 13-15, $\iota$ is a non-even hook-type orbit $(2n-2m,1^{2m+1})$, $L=\GL_{2m+1}\times \GL_1^{2n-2m}$, and  $\Delta_{red}=(\GL_{2m+1},\GL_{2m+1},0,\rho_H\oplus T(std_{\GL_{2m+1}}))$ and they are all vector space cases.
    \begin{itemize}
        \item For Quadruple 13, $\rho_H=0$, and $\hat\Delta_{red}=(\GL_{2m+1},\GL_{1},(2m,1),0)$. 
        By Conjecture \ref{Whittaker induction}, $\hat\Delta = (\GL_{2n+1},\GL_{2n-2m+1},(2m, 1^{2n-2m+1}),0)$  which is just Quadruple 13 except replacing $m$ with $n-m$.
        \item For Quadruple 14, $\rho_H=T(std_{\GL_{2m+1}})$, and
        $\hat\Delta_{red}=(\GL_{2m+1},\GL_{2},(2m-1,1^2),T(std_{\GL_{2}}))$.
        By Conjecture \ref{Whittaker induction}, $\hat\Delta=(\GL_{2n+1},\GL_{2n-2m+2},(2m-1, 1^{2n-2m+2}),T(std_{\GL_{2n-2m+2}}))$ which is Quadruple 17.
        \item For Quadruple 15, $\rho_H=T(\wedge^2)$, and $\hat\Delta_{red}=(\GL_{2m+1},\GL_{m+1}\times\GL_{m},1,T(std_{\GL_{m+1}}))$. 
        By Conjecture \ref{Whittaker induction}, $\hat\Delta=(\GL_{2n+1}, \GL_{2n-m+1}\times \GL_m, 1, T(std_{\GL_{2n-2m+1}}))$ which is Quadruple 7.
    \end{itemize}

    \item For Quadruples 16-19, $\iota$ is an even hook-type orbit $(2n-2m+1,1^{2m})$, $L=\GL_{2m+1}\times \GL_1^{2n-2m}$ and $\rho_\iota$ is trivial. For Quadruples 16-18, $\Delta_{red}=(\GL_{2m+1},\GL_{2m},1,\rho_H)$, while for Quadruple 19, $\Delta_{red}=(\GL_{2m+1},\Sp_{2m},1,0)$. They are all polarized cases.
    \begin{itemize}
        \item For Quadruple 16, $\rho_H=0$ and $\hat\Delta_{red}=(\GL_{2m+1},\GL_2,(2m-1,1^2),0)$. 
        By Conjecture \ref{Whittaker induction}, $\hat\Delta=(\GL_{2n+1},\GL_{2n-2m+2},(2m-1,1^{2n-2m+2}),0)$ which is just Quadruple 16 except replacing $m$ with $n-m+1$. 
        \item For Quadruple 17, $\rho_H=T(std_{\GL_{2m}})$, and $\hat\Delta_{red}=(\GL_{2m+1},\GL_3,(2m-2, 1^3),T(std_{\GL_3}))$. 
        By Conjecture \ref{Whittaker induction}, $\hat\Delta=(\GL_{2n+1},\GL_{2n-2m+3},(2m-2,1^{2n-2m+3}),T(std_{\GL_{2n-2m+3}}))$ which is Quadruple 14.
        \item For Quadruple 18, $\rho_H=T(\wedge^2)$ and $\hat\Delta_{red}=(\GL_{2m+1},\GL_{m+1}\times\GL_m,1,T(std_{\GL_m}))$.
        By Conjecture \ref{Whittaker induction}, $\hat\Delta=(\GL_{2n+1},\GL_{2n-m+1}\times \GL_m,1,T(std_{\GL_m}))$ which is Quadruple 8.
        \item For Quadruple 19, $\hat\Delta_{red}=(\GL_{2m+1},\GL_{m+1}\times\GL_m,1,0)$.
        By Conjecture \ref{Whittaker induction}, $\hat\Delta=(\GL_{2n+1},\GL_{2n-m+1}\times \GL_m,1,0)$ which is Quadruple 6.
    \end{itemize}

    \item For Quadruple 20, $\iota=(2n-5,1^6)$, $L=\GL_7\times \GL_1^{2n-6}$, and
    $\Delta_{red}=(\GL_7,\GL_6,1,\wedge^3)$. Although we cannot compute the dual of $\Delta_{red}$ since it is neither polarized or a vector space case, by our argument for Quadruple 11 we know that the dual of $\Delta$ should be Quadruple 11.

    \item For Quadruple 21, this is the Whittaker model case, and it is dual to $(\GL_{2n+1},\GL_{2n+1},1,0)$ which is Quadruple 1.
\end{itemize}

Now we consider the cases where $G=\GL_{2n}$.
\begin{itemize}
    \item For Quadruples 1-5, they are the vector space cases. We refer to \cite{MWZ2} for their duals:
    
    \begin{itemize}
    \item Quadruple 1 is dual to $(\GL_{2n},1,(2n),0)$, which is Quadruple 23.
    \item Quadruple 2 is dual to $(\GL_{2n},\GL_1,(2n-1,1),0)$ which is Quadruple 15 with $m=0$.
    \item Quadruple 3 is dual to $(\GL_{2n},\GL_{n},(2^n),T(std_{\GL_n}))$, which is  Quadruple 14.
    \item Quadruple 4 is dual to $(\GL_{2n},\GL_2,(2n-2,1^2),T(std_{\GL_2}))$ when $n>1$ (this is Quadruple 19 with $m=1$). When $n=1$ it is self-dual.
    \item Quadruple 5 is dual to $(\GL_{2n},\GL_{n}\times\GL_n,1,T(std_{\GL_n}))$ which is Quadruple 7 with $m=n$.
    \end{itemize}
    
    \item For Quadruples 6-11, $\iota$ is the trivial orbit, and they are all polarized cases. We can use the algorithm in \cite{BSV} to compute the dual. 

    \begin{itemize}
        \item Quadruple 6 is dual to $(\GL_{2n}, \Sp_{2m}, (2n-2m, 1^{2m}), 0)$, which is   Quadruple 21.
        \item For Quadruple 7, when $m<n$, it is dual to $(\GL_{2n}, \GL_{2m+1}, (2n-2m-1,1^{2m+1}), T(\wedge^2))$, which is   Quadruple 17 (when $m=n-1$ this is just Quadruple 9). When $m=n$, it is dual to $(\GL_{2n},\GL_{2n},1,T(std\oplus \wedge^2))$, which is Quadruple 5.
        \item For Quadruple 8, the case when $m=n$ has already been considered in Quadruple 7. When $m<n$, the dual is $(\GL_{2n}, \GL_{2m}, (2n-2m,1^{2m}), T(\wedge^2))$, which is Quadruple 20.
        \item Quadruple 9 is dual to $(\GL_{2n}, \GL_{n+1}\times \GL_{n-1}, 1, T(std_{\GL_{n+1}}))$, which is Quadruple 7 with $m=n-1$.
        \item Quadruple 10 is dual to $(\GL_{2n}, \GL_n, (2^{n}), 0)$, which is   Quadruple 13.
        \item Quadruple 11 is dual to $(\GL_{2n}, \GL_n\times\GL_n, (1^{2n}), 0)$ which is Quadruple 6 with $m=n$.
    \end{itemize}

    \item For Quadruple 12, $\iota=(2^2,1^{2n-4})$, $L=\GL_2^{ 2}\times \GL_{2n-4}$, and $\Delta_{red}=(\GL_2^{ 2}\times \GL_{2n-4},\GL_{2}\times\GL_{2n-4},0,T(std_{\GL_{2}}\otimes std_{\GL_{2n-4}}))$. This is a polarized case, whose dual is $\hat\Delta_{red}=(\GL_2^{ 2}\times \GL_{2n-4},\GL_{2}^{ 2}\times \GL_{2},1^2\times(2n-6,1^2), std_{\GL_{2}}\otimes std_{\GL_{2}} \otimes std_{\GL_{2}})$. By Conjecture \ref{Whittaker induction}, $\hat\Delta$ is $(\GL_{2n},\GL_{6},(2n-6,1^6),\wedge^3)$, which is Quadruple 22.
    
    \item For Quadruples 13-14, $\iota=(2^n)$, $L=\GL_n\times \GL_n$ and $\rho_\iota$ is trivial. In these two cases, $\Delta_{red}=(\GL_n\times\GL_n,\GL_n,1,\rho_H)$ and they are both polarized. 
    \begin{itemize}
        \item For Quadruple 13, $\rho_H=0$, and $\hat\Delta_{red}=(\GL_n\times\GL_n,\GL_n,1,0)$.
        By Conjecture \ref{Whittaker induction}, $\hat\Delta=(\GL_{2n},\Sp_{2n},1,0)$, which is Quadruple 10.
        
        \item For Quadruple 14, $\rho_H=T(std_{\GL_n})$, and $\hat\Delta_{red}=(\GL_n\times\GL_n,\GL_n\times \GL_n,1,T(std_{\GL_n}\otimes std_{\GL_n}))$.
        By Conjecture \ref{Whittaker induction}, $\hat\Delta=(\GL_{2n},\GL_{2n},1,T(\wedge^2))$, which is  Quadruple 3.
    \end{itemize}
    
    \item For Quadruples 15-17, $\iota=(2n-2m-1,1^{2m+1})$, $L=\GL_{2m+2}\times \GL_1^{2n-2m-2}$ and $\rho_\iota$ is trivial. In these cases, $\Delta_{red}=(\GL_{2m+2},\GL_{2m+1},0,\rho_H)$ and are all polarized cases.
    \begin{itemize}
        \item For Quadruple 15, $\rho_H=0$, 
        and $\hat\Delta_{red}=(\GL_{2m+2},\GL_2,(2m,1^2),0)$.
        By Conjecture \ref{Whittaker induction}, $\hat\Delta=(\GL_{2n},\GL_{2n-2m},(2m,1^{2n-2m}),0)$, which is Quadruple 18.
        
        \item For Quadruple 16, $\rho_H=T(std_{\GL_{2m+1}})$,
        and $\hat\Delta_{red}=(\GL_{2m+2},\GL_3,(2m-1,1^3),T(std_{\GL_3}))$.
        By Conjecture \ref{Whittaker induction}, $\hat\Delta=(\GL_{2n},\GL_{2n-2m+1},(2m-1,1^{2n-2m+1}),T(std_{\GL_{2n-2m+1}}))$ which is just Quadruple 16 except replacing $m$ with $n-m$.
        
        \item For Quadruple 17, $\rho_H=T(\wedge^2)$,
        and $\hat\Delta_{red}=(\GL_{2m+2},\GL_{m+2}\times\GL_m,1,T(std_{\GL_{m+2}}))$.
        By Conjecture \ref{Whittaker induction}, $\hat\Delta=(\GL_{2n},\GL_{2n-m}\times \GL_m,1,T(std_{\GL_{2n-m}}))$ which is Quadruple 7.
    \end{itemize}
    
    \item For Quadruples 18-21, $\iota=(2n-2m,1^{2m})$, $L=\GL_{2m}\times \GL_1^{2n-2m}$ and $\rho_\iota=T(std_{\GL_{2m}})$. For Quadruples 18-20, $\Delta_{red}=(\GL_{2m},\GL_{2m},1,\rho_H\oplus T(std_{\GL_{2m}}))$, which are vector space cases. For Quadruple 21, $\Delta_{red}=(\GL_{2m},\Sp_{2m},1,T(std_{\Sp_{2m}}))$, which is a polarized case.
    \begin{itemize}
        \item For Quadruple 18, $\rho_H=0$,
        and $\hat\Delta_{red}=(\GL_{2m},\GL_1,(2m-1,1),0)$.
        By Conjecture \ref{Whittaker induction}, $\hat\Delta=(\GL_{2n},\GL_{2n-2m+1},(2m-1,1^{2n-2m+1}),0)$, which is Quadruple 15.
        
        \item For Quadruple 19, $\rho_H=T(std_{\GL_{2m}})$,
        and $\hat\Delta_{red}=(\GL_{2m},\GL_2,(2m-2,1^2),T(std_{\GL_2}))$.
        By Conjecture \ref{Whittaker induction}, $\hat\Delta=(\GL_{2n},\GL_{2n-2m+2},(2m-2,1^{2n-2m+2}),T(std_{\GL_{2n-2m+2}}))$ which is just Quadruple 16 except replacing $m$ with $n-m+1$.
        
        \item For Quadruple 20, $\rho_H=T(\wedge^2)$,
        and $\hat\Delta_{red}=(\GL_{2m},\GL_{m}\times\GL_m,1,T(std_{\GL_m}))$.
        By Conjecture \ref{Whittaker induction}, $\hat\Delta=(\GL_{2n},\GL_{2n-m}\times \GL_m,(2m,1^{2n-2m}),T(std_{\GL_m}))$ which is Quadruple 8.
        
        \item For Quadruple 21, $\hat\Delta_{red}=(\GL_{2m},\GL_{m}\times\GL_m,1,0)$. By Conjecture \ref{Whittaker induction}, $\hat\Delta=(\GL_{2n},\GL_{2n-m}\times \GL_m,(1^{2n}),0)$ which is Quadruple 6.
    \end{itemize}
    
    \item For Quadruple 22, $\iota=(2n-6,1^6)$, $L=\GL_6\times \GL_1^{2n-6}$, and 
    $\Delta_{red}=(\GL_6,\GL_6,1,\wedge^3\oplus T(std_{\GL_6}))$. This is a vector space case, whose dual is $\hat\Delta_{red}=(\GL_6,\GL_2\times \GL_2,(2^2,1^2),0)$. By Conjecture \ref{Whittaker induction}, $\hat\Delta=(\GL_{2n},\GL_2\times \GL_{2n-4},(2^2,1^{2n-4}),0)$, which is   Quadruple 12.
    
    \item For Quadruple 23, this is the Whittaker model case and it is dual to $(\GL_{2n},\GL_{2n},1,0)$ which is Quadruple 1.

    \item For Quadruple 24, it is the vector space case and by \cite{MWZ2} we know that its dual is $(\GL_6,\GL_2,(3^2),0)$ which is Quadruple 26.
    
    \item For Quadruple 25, it is the vector space case and by \cite{MWZ2} we know that its dual is $(\GL_6, \GL_2\times\GL_2, (2^2,1^2),0)$ which is Quadruple 12 with $n=3$.
    
    \item For Quadruple 26, $\iota=(3^2)$, $L=\GL_2^{ 3}$, and  $\Delta_{red}=(\GL_2^{ 3},\GL_2,1,0)$. This is a polarized case and its dual is $\hat\Delta_{red}=(\GL_2^{ 3},\GL_2^{ 3},T(std_{\GL_2}\otimes std_{\GL_2}\otimes std_{\GL_2}))$.
    By Conjecture \ref{Whittaker induction}, $\hat\Delta=(\GL_{6},\GL_6,(1^6),\wedge^3)$, which is Quadruple 24.

     \item For Quadruple 27, our current method cannot be used to compute its dual. However, since this is the only remaining case, we believe it is self-dual \footnote{Among all the quadruples with $G$ simple, this is the only case where we cannot provide any evidence for the duality.}.
\end{itemize}

We summarize our results for Type $A$ in Table \ref{Table A} \footnote{we combine the even and odd cases into one table}.

\begin{rmk}
By a similar argument as in this section, we can also classify the quadruples and compute their dual in Type $B$, $C$, $D$ cases. We will skip the details and only summarize our results in Table \ref{Table B}, \ref{Table C} and \ref{Table D}. We refer the reader to the Ph.D. thesis \cite{T} of the first author for details of those cases.
\end{rmk}

\section{Type $E_8$ case} \label{sec:typeE8}
\subsection{The Classification}
The SLA package \cite{GT} in GAP provides a representative of each nilpotent orbit in $E_8$, which we use to compute the Levi subalgebra $\Fl$, $\Fg_\iota$, and $\Fu/\Fu^+$. In the table below, the orbits are listed by their Bala-Carter labels, with the Levi-spherical ones marked by a bullet on its left. In certain cases, underlines indicate how $\Fh_0$ is embedded into $\Fl$.

\begin{table}[htbp]
\centering
\begin{minipage}[t]{0.45\textwidth}
\centering
\begin{tabular}[t]{cl|cc}
\hline
 & $\iota$ & $\mathfrak{[l,l]}$ & $\Fg_\iota$\\
\hline 
$\bullet$  & $1$ & $\mathfrak{e}_{8}$  & $\mathfrak{e}_{8}$ \\
$\bullet$  & $A_{1}$ & $\mathfrak{e}_{7}$  & $\mathfrak{e}_{7}$ \\
$\bullet$  & $2A_{1}$ & $\Fso_{14}$  & $\Fso_{13}$ \\
$\bullet$  & $3A_{1}$ & $\mathfrak{e}_{6}\oplus\Fsl_{2}$  & $\mathfrak{f}_{4}\oplus\Fsl_{2}$ \\
 & $A_{2}$ & $\mathfrak{e}_{7}$  & $\mathfrak{e}_{6}$\\
$\bullet$  & $4A_{1}$ & $\Fsl_{8}$  & $\Fsp_{8}$\\
$\bullet$  & $A_{2}+A_{1}$ & $\Fso_{12}$  & $\Fsl_{6}$ \\
 & $A_{2}+2A_{1}$ & $\Fso_{10}\oplus\Fsl_{3}$  & $\Fspin_{7}\oplus\Fso_{3}$ \\
$\bullet$  & $A_{3}$ & $\Fso_{12}$  & $\Fso_{\ensuremath{11}}$ \\
 & $A_{3}$ & $\Fsl_{7}\oplus\Fsl_{2}$  & $\mathfrak{g}_{2}\oplus\Fsl_{2}$ \\
 & $2A_{2}$ & $\Fso_{14}$  & $\mathfrak{g}_{2}\oplus\mathfrak{g}_{2}$ \\
 & $2A_{2}+A_{1}$ & $\Fso_{10}\oplus\Fsl_{2}$  & $\mathfrak{g}_{2}\oplus\Fsl_{2}$ \\
 & $A_{3}+A_{1}$ & $\Fso_{10}\oplus\Fsl_{2}$  & $\Fspin_{7}\oplus\Fsl_{2}$ \\
 & $D_{4}(a_{1})$ & $\mathfrak{e}_{6}\oplus\Fsl_{2}$  & $\Fso_{8}$ \\
$\bullet$  & $D_{4}$ & $\mathfrak{e}_{6}$ & $\mathfrak{f}_{4}$ \\
 & $2A_{2}+2A_{1}$ & $\Fsl_{4}\oplus\Fsl_{5}$  & $\Fsp_{4}$ \\
 & $A_{3}+2A_{1}$ & $\Fsl_{6}\oplus\Fsl_{2}$  & $\Fsp_{4}\oplus\Fsl_{2}$ \\
 & $D_{4}(a_{1})+A_{1}$ & $\Fsl_{6}\oplus\Fsl_{2}$  & $\Fsl_{2}^{\oplus3}$ \\
 & $A_{3}+A_{2}$ & $\Fso_{8}\oplus\Fsl_{3}$  & $\Fsp_{4}\oplus\Ft^{1}$ \\
 & $A_{4}$ & $\Fso_{12}$  & $\Fsl_{5}$ \\
 & $A_{3}+A_{2}+A_{1}$ & $\Fsl_{5}\oplus\Fsl_{3}\oplus\Fsl_{2}$  & $\Fsl_{2}\oplus\Fsl_{2}$ \\
$\bullet$  & $D_{4}+A_{1}$ & $\Fsl_{6}$  & $\Fsp_{6}$ \\
 & $D_{4}(a_{1})+A_{2}$ & $\Fsl_{8}$  & $\Fsl_{3}$ \\
 & $A_{4}+A_{1}$ & $\Fso_{8}\oplus\Fsl_{2}$  & $\Fsl_{3}\oplus\Ft^{1}$ \\
 & $2A_{3}$ & $\Fsl_{4}^{\oplus2}$  & $\Fsp_{4}$ \\
 & $D_{5}(a_{1})$ & $\Fso_{8}\oplus\Fsl_{2}$  & $\Fsl_{4}$ \\
 & $A_{4}+2A_{1}$ & $\Fsl_{3}\oplus\Fsl_{4}\oplus\Fsl_{2}$  & $\Fsl_{2}\oplus\Ft^{1}$ \\
 & $A_{4}+A_{2}$ & $\Fso_{10}\oplus\Fsl_{3}$  & $\Fsl_{2}^{\oplus2}$ \\
$\bullet$  & $A_{5}$ & $\Fso_{8}\oplus\Fsl_{2}$  & $\Fg_2\oplus\Fsl_{2}$ \\
 & $D_{5}(a_{1})+A_{1}$ & $\Fsl_{4}\oplus\Fsl_{2}\oplus\Fsl_{3}$  & $\Fsl_{2}^{\oplus2}$ \\
 & $A_{4}+A_{2}+A_{1}$ & $\Fsl_{4}\oplus\Fsl_{3}\oplus\Fsl_{2}$  & $\Fsl_{2}$ \\
 & $D_{4}+A_{2}$ & $\Fsl_{7}$  & $\Fsl_{3}$ \\
 & $E_{6}(a_{3})$ & $\Fso_{10}\oplus\Fsl_{2}$  & $\Fg_2$\\
 & $D_{5}$ & $\Fso_{10}$  & $\Fspin_{7}$ \\
 & $A_{4}+A_{3}$ & $\Fsl_{3}^{\oplus2}\oplus\Fsl_{2}^{\oplus2}$  & $\Fsl_{2}$ \\
\hline
\end{tabular}
\end{minipage}
\hfill
\begin{minipage}[t]{0.45\textwidth}
\centering
\begin{tabular}[t]{cl|cc}
\hline
& $\iota$ & $\mathfrak{[l,l]}$ & $\Fg_\iota$  \\
\hline 
 & $A_{5}+A_{1}$ & $\Fsl_{4}\oplus\Fsl_{2}^{\oplus2}$  & $\Fsl_{2}^{\oplus2}$ \\
 & $D_{5}(a_{1})+A_{2}$ & $\Fsl_{4}\oplus\Fsl_{2}^{\oplus2}$  & $\Fsl_{2}$ \\
 & $D_{6}(a_{2})$ & $\Fsl_{4}\oplus\Fsl_{2}^{\oplus2}$  & $\Fsl_{2}^{\oplus2}$ \\
 & $E_{6}(a_{3})+A_{1}$ & $\Fsl_{4}\oplus\Fsl_{2}^{\oplus2}$  & $\Fsl_{2}$ \\
 & $D_{7}(a_{5})$ & $\Fsl_{3}^{\oplus2}\oplus\Fsl_{2}^{\oplus2}$  & $\Fsl_{2}$ \\
 & $D_{5}+A_{1}$ & $\Fsl_{4}\oplus\Fsl_{2}$  & $\Fsl_{2}^{\oplus2}$ \\
 & $E_{8}(a_{7})$ & $\Fsl_{5}\oplus\Fsl_{4}$  & $0$\\
 & $A_{6}$ & $\Fso_{8}\oplus\Fsl_{3}$  & $\Fsl_{2}^{\oplus2}$ \\
 & $D_{6}(a_{1})$ & $\Fsl_{4}\oplus\Fsl_{2}$  & $\Fsl_{2}^{\oplus2}$ \\
 & $A_{6}+A_{1}$ & $\Fsl_{3}\oplus\Fsl_{2}^{\oplus3}$  & $\Fsl_{2}$ \\
 & $E_{7}(a_{4})$ & $\Fsl_{3}\oplus\Fsl_{2}^{\oplus3}$  & $\Fsl_{2}$ \\
 & $E_{6}(a_{1})$ & $\Fso_{8}\oplus\Fsl_{2}$  & $\Fsl_{3}$ \\
 & $D_{5}+A_{2}$ & $\Fsl_{5}\oplus\Fsl_{3}$  & $\Ft^{1}$\\
$\bullet$  & $D_{6}$ & $\Fsl_{4}$  & $\Fsp_{4}$ \\
$\bullet$  & $E_{6}$ & $\Fso_{8}$  & $\Fg_2$\\
 & $D_{7}(a_{2})$ & $\Fsl_{2}^{\oplus4}$  & $\Ft^{1}$\\
 & $A_{7}$ & $\Fsl_{2}^{\oplus4}$  & $\Fsl_{2}$ \\
 & $E_{6}(a_{1})+A_{1}$ & $\Fsl_{2}^{\oplus4}$  & $\Ft^{1}$\\
 & $E_{7}(a_{3})$ & $\Fsl_{2}^{\oplus4}$  & $\Fsl_{2}$ \\
 & $E_{8}(b_{6})$ & $\Fsl_{4}\oplus\Fsl_{3}\oplus\Fsl_{2}$  & $0$\\
 & $D_{7}(a_{1})$ & $\Fsl_{4}\oplus\Fsl_{3}$  & $\Ft^{1}$\\
 & $E_{6}+A_{1}$ & $\Fsl_{2}\oplus\underline{\Fsl_{2}}\oplus\underline{\Fsl_{2}}$  & $\underline{\Fsl_{2}}$ \\
 & $E_{7}(a_{2})$ & $\Fsl_{2}\oplus\Fsl_{2}\oplus\underline{\Fsl_{2}}$  & $\underline{\Fsl_{2}}$ \\
 & $E_{8}(a_{6})$ & $\Fsl_{3}^{\oplus2}\oplus\Fsl_{2}^{\oplus2}$  & $0$\\
$\bullet$ & $D_{7}$ & $\underline{\Fsl_{2}}\oplus\underline{\Fsl_{2}}$  & $\underline{\Fsl_{2}}$ \\
 & $E_{8}(b_{5})$ & $\Fsl_{3}\oplus\Fsl_{2}\oplus\Fsl_{3}$  & $0$\\
 & $E_{7}(a_{1})$ & $\Fsl_{2}\oplus\underline{\Fsl_{2}}$  & $\underline{\Fsl_{2}}$ \\
 & $E_{8}(a_{5})$ & $\Fsl_{3}\oplus\Fsl_{2}^{\oplus3}$  & $0$\\
 & $E_{8}(b_{4})$ & $\Fsl_{3}\oplus\Fsl_{2}^{\oplus2}$  & $0$\\
$\bullet$  & $E_{7}$ & $\Fsl_{2}$  & $\Fsl_{2}$ \\
 & $E_{8}(a_{4})$ & $\Fsl_{2}^{\oplus4}$  & $0$\\
 & $E_{8}(a_{3})$ & $\Fsl_{2}^{\oplus3}$  & $0$\\
 & $E_{8}(a_{2})$ & $\Fsl_{2}^{\oplus2}$  & $0$\\
 & $E_{8}(a_{1})$ & $\Fsl_{2}$  & $0$\\
$\bullet$  & $E_{8}$ & $0$ & $0$\\
\hline 
\end{tabular}
\end{minipage}
\label{Table E_8 Levi-spherical}
\caption{The sphericality of nilpotent orbits of $E_8$.}
\end{table}

For Levi-spherical orbits, we check whether $\Fu/\Fu^+$ is a subrepresentation of an anomaly-free coisotropic representation of $\Fg_\iota$ and $\Fg_\iota'$, where $\Fg_\iota'$ is the Lie algebra of the stabilizer (in $G_\iota$) of a generic point of $\Fg_{\iota}^{\perp}$. We leave the last column empty when $\Fu/\Fu^+$ does not satisfy the condition for $\Fg_\iota$. 
Same as the previous table, we put a bullet on the left of those orbits that satisfy the condition for both $\Fg_\iota$ and $\Fg_\iota'$.

\begin{table}[htbp]
\centering
\begin{tabular}{ll|ccc|cc}
\hline 
 & $\iota$ & $\mathfrak{[l,\mathfrak{l}]}$ & $\Fg_\iota$ & $\Fg_\iota'$ & $\Fu/\Fu^+|_{\Fg_\iota}$ & $\Fu/\Fu^+|_{\Fg_\iota'}$\\
\hline 
$\bullet$ & $1$ & $\mathfrak{e}_{8}$ & $\mathfrak{e}_{8}$ & $\mathfrak{e}_{8}$ & $0$ & $0$\\
$\bullet$ & $A_{1}$ & $\mathfrak{e}_{7}$ & $\Fe_7$ & $\mathfrak{e}_{7}$ & $std_{\mathfrak{e}_{7}}$ & $std_{\mathfrak{e}_{7}}$\\
 & $2A_{1}$ & $\Fso_{14}$ & $\Fso_{13}$ & $\Fso_{12}$ & $\Fspin_{13}$ & \\
 & $3A_{1}$ & $\mathfrak{e}_{6}\oplus\Fsl_{2}$ & $\Ff_{4}\oplus\Fsl_{2}$ & $\Fso_{8}\oplus\Fsl_{2}$ & $\mathfrak{f}_{4}\otimes\Fsl_{2}\oplus\Fsl_{2}$ & \\
 & $4A_{1}$ & $\Fsl_{8}$ & $\Fsp_{8}$ & $\Fsl_{2}^{\oplus4}$ & $std\oplus\wedge^{3}$ & \\
 & $A_{2}+A_{1}$ & $\Fso_{12}$ & $\Fsl_{6}$ & $\Fsl_{2}^{\oplus3}$ & $T(std)^{\oplus2}\oplus\wedge^{3}$ & \\
 & $A_{3}$ & $\Fso_{12}$ & $\Fso_{\ensuremath{11}}$ & $\Fso_{10}$ & $\Fspin_{11}$ & \\
$\bullet$ & $D_{4}$ & $\mathfrak{e}_{6}$ & $\mathfrak{f}_{4}$ & $\Fso_{8}$ & $0$ & $0$\\
 & $D_{4}+A_{1}$ & $\Fsl_{6}$ & $\Fsp_{6}$ & $\Fsl_{2}^{\oplus3}$ & $std\oplus\wedge^{3}$ & \\
 & $A_{5}$ & $\Fso_{8}\oplus\Fsl_{2}$ & $\Fg_2\oplus\Fsl_{2}$ & $\Fsl_{2}\oplus\Fsl_{2}$ & $\Fg_2\otimes\Fsl_{2}\oplus\Fsl_{2}$ & \\
 & $D_{6}$ & $\Fsl_{4}$ & $\Fsp_{4}$ & $\Fsl_{2}\oplus\Fsl_{2}$ & $\Fsp_{4}^{\oplus3}$ & \\
$\bullet$ & $E_{6}$ & $\Fso_{8}$ & $\Fg_2$ & $\Fsl_{2}$ & $0$ & $0$\\
 & $D_{7}$ & $\Fsl_{2}\oplus\Fsl_{2}$ & $\Fsl_{2}$ & $\Ft^{2}$ & $std^{\oplus5}$ & \\
 & $E_{7}$ & $\Fsl_{2}$ & $\Fsl_{2}$ & $\Fsl_{2}$ & $std^{\oplus3}$ &  \\
$\bullet$ & $E_{8}$ & $0$ & $0$ & $0$ & $0$ & $0$\\
\hline 
\end{tabular}

\label{Table E_8 coisotropic}
\caption{The multiplicity-freeness of Levi-spherical nilpotent orbits of $E_8$.}
\end{table}

\begin{rmk}
When $\iota=2A_1$, $A_3$, or $E_7$, the representation $\Fu/\Fu^+$ needs an extra $\Fgl_1$ action to be coisotropic as a representation of $\Fg_\iota'$.
\end{rmk}

Next, we study the remaining 5 orbits. 

\begin{itemize}
    \item When $\iota=1$, $H$ can be any reductive spherical subgroup of $E_8$, which are $E_8,E_7\times \SL_2$ or $\SO_{16}$. However, when $H=E_7\times \SL_2$ or $\SO_{16}$, the corresponding spherical variety has Type N spherical root, as defined in \cite{BP}, and hence the generic stabilizer of the Hamiltonian space is not connected. Hence $H$ has to be $E_8$. Since there is no nontrivial multiplicity-free symplectic representation of $E_8$, we know that $\rho_H$ must be $0$.
    \item When $\iota=A_1$, $L$ and $G_\iota$ are both of Type $E_7$ and $\Fu/\Fu^+$ is the standard representation of $E_7$. In this case, $H$ must be a reductive spherical subgroup of $E_7$, which is of Type $E_7,E_6,A_1\times D_6$ or $A_7$. When $H$ is of Type $E_6, A_1\times D_6$ or $A_7$, the restriction of the standard representation of $E_7$ to the generic stabilizer $H'$ is not coisotropic. Hence $H$ has to be equal to $G_\iota$. In this case, since the only no nontrivial multiplicity-free symplectic representation of $E_7$ is the standard representation, we must have $\rho_H=0$.
    \item When $\iota=D_4$, $L$ is of Type $E_6$ and $G_\iota$ is of Type $F_4$. Then $H$ has to equal to $G_\iota$ since $E_6$ does not have any reductive spherical subgroup that is a proper subgroup of $F_4$. Moreover, since there is no nontrivial multiplicity-free symplectic representation of $F_4$, we know that $\rho_H$ must be $0$.
    \item When $\iota=E_6$, $L$ is of Type $D_4$ and $G_\iota$ is of Type $G_2$. Then $H$ has to equal to $G_\iota$ since $\SO_8$ does not have any reductive spherical subgroup that is a proper subgroup of $G_2$. As for $\rho_H$, the only multiplicity-free symplectic representation of $G_2$ are $0$ and $T(std_{G_2})$. As the restriction of $T(std_{G_2})$ to the generic stabilizer $G_\iota'=\SL_2$ is not coisotropic, we know that $\rho_H$ must be $0$.
    \item When $\iota=E_8$ is the regular orbit, it is easy to see that $H$ has to be trivial and $\rho_H$ has to be $0$.
\end{itemize}

Summarizing, we get the following table of anomaly-free hyperspherical BZSV quadruples for $E_8$.

\begin{table}[htbp]
\centering
\begin{tabular}{|c|c|c|c|c|c|}
\hline 
\textnumero & $\Delta$ & $[L,L]$ & $H'$ & $\Fu/\Fu^{+}|_H$ & $\rho_{H,\iota}|_{H'}$ \\
\hline 
1 & $(E_{8},E_{8},0,0)$ & $E_{8}$ & $E_{8}$ & $1$ & $0$ \\
\hline 
2 & $(E_{8},E_{7},A_{1}, 0)$ & $E_{7}$ & $E_{7}$ & $std_{E_7}$ & $std_{E_7}$ \\
\hline 
3 & $(E_{8},F_{4},D_{4},0)$ & $E_{6}$ & $\SO_{8}$ & $0$ & $0$ \\
\hline 
4 & $(E_{8},G_{2},E_{6},0)$ & $\SO_{8}$ & $\SL_{2}$ & $0$ & $0$ \\
\hline 
5 & $(E_{8},0,E_{8},0)$ & $0$ & $0$ & $0$ & $0$ \\
\hline 
\end{tabular}
\caption{Type $E_8$}
\label{Table Type E_8}
\end{table}

\subsection{The Dual}
We consider the dual for each quadruple:

\begin{itemize}
    \item For Quadruple 1, this is the trivial orbit case. The dual is the Quadruple 5.
    \item For Quadruple 2, $\Delta_{red} = (E_7, E_7, 0, std)$ (which is a vector space), whose dual is $\hat\Delta_{red}=(E_7,\SL_2,E_6,0)$. By Conjecture \ref{Whittaker induction}, the dual quadruple is Quadruple 4.
    \item For Quadruple 3, $\Delta_{red} = (E_6, F_4, 0, 0)$ (which is a polarized case), whose dual is $\hat\Delta_{red}=(E_6,\SL_3,D_4,0)$. By Conjecture \ref{Whittaker induction}, the dual quadruple is itself
    \item For Quadruple 4, $\Delta_{red} = (\SO_8,G_2,0,0)$ (which is a polarized case), whose dual is $\hat\Delta_{red}=(\SO_8,\SO_4\times \Sp_2,(2^2,1^4),0)$. By Conjecture \ref{Whittaker induction}, the dual quadruple is Quadruple 2.
    \item For Quadruple 5, this is the Whittaker model case. The dual is  Quadruple 1.
\end{itemize}

We summarize our results for $E_8$ in Table \ref{Table E8}.

\begin{rmk}
By a similar argument as in this section, we can also classify the quadruples and compute their dual in Type $F_4$, $E_6$, $E_7$ cases. We will skip the details and only summarize our results in Table \ref{Table F_4}, \ref{Table E_6} and \ref{Table E_7}. We refer the reader to the PhD thesis of the first author for details of those cases.
\end{rmk}

\section{Tables} \label{sec:tables}
We summarize our finding in 10 tables, each of them corresponds to one of the following root types: $B_2=C_2,G_2,F_4,E_6,E_7,E_8,A_n,B_n,C_n,D_n$.  For Type $A$, $B_2=C_2$, $D$, $E$, $F$, $G$, each quadruple will appear twice in the table, one on  $\Delta$-side and one on $\hat{\Delta}$-side. For Type $B$ (resp. Type $C$), each quadruple will appear once in the table of Type $B$ (resp. Type $C$) on the $\Delta$-side and once in the Table of Type $C$ (resp. Type $B$) on the $\hat{\Delta}$-side.

\begin{rmk} In some cases, one needs to use groups like $\Spin,\GSO,\GSp,\SL$ instead of $\SO,\Sp,\GL$. Here for simplicity we are ignoring this. The groups in the table has the same root type as the ``correct" choice, they may be differed by a finite isogeny and/or some $\GL_1$-factors. Also in Type D, everything may be differed by the outer automorphism.

The two ``type" columns indicate the evidence for duality: The first column corresponds to the duality from $\Delta$ to $\hat{\Delta}$, while the second applies to the reverse. We use the following notation:
\begin{enumerate}
    \item P: if the dual quadruple is determined by the algorithm of polarized cases in \cite{BSV};
    \item V: if the quadruple is a vector space case. Its dual is given in \cite{MWZ1};
    \item W: if the dual quadruple is verified through Conjecture \ref{Whittaker induction} (i.e. it is the Whittaker induction of a polarized reductive quadruple/vector space).
    \item $\ast$: this only happens in Type D, for those cases we need to use some other argument to compute the dual of $\Delta_{red}$ and then use Conjecture \ref{Whittaker induction} to compute the dual of $\Delta$. We refer the reader to the last part of this section for details.
    \item ?: we do not have any evidence \footnote{the only case where both ``type" columns are ? is the quadruple $\Delta=\hat{\Delta}=(\GL_8,\GL_6\times \GL_2,1,\wedge^{3}_{\GL_6})$, this is the only case where we do not have any evidence of the duality for both directions.}.
\end{enumerate}
\end{rmk}

\begin{figure}[h!]
	\begin{tabular}{| c | c | c | c | c |}
		\hline
		\textnumero & $\Delta=(G,H,\iota,\rho_H)$ & \multicolumn{2}{c|}{Type} & $\hat{\Delta}=(\hat{G},\hat{H}',\hat{\iota}',\rho_{\hat{H}'})$  \\
		\hline
		1 & $(\Sp_4,\Sp_4,1,0)$ & PV & W & $(\Sp_4,1,(4),0)$ \\
		\hline
        2 & $(\Sp_4,\Sp_4,1,T(std_{\Sp_4}))$ & PV & W & $(\Sp_4,\GL_1,(2^2),0)$ \\
        \hline
        3 & $(\Sp_4,\Sp_4,1,T(std_{\Sp_4}\oplus std_{\Sp_4}))$ & PV & P & $(\Sp_4,\Sp_2\times \Sp_2,1,T(std_{\SL_2,1}\oplus std_{\SL_2,2}))$ \\
        \hline
        4 & $(\Sp_4,\GL_1\times \Sp_2,1,0)$ & P & P & $(\Sp_4,\Sp_2\times \Sp_2,1,T(std_{\SL_2}))$ \\
        \hline
        5 & $(\Sp_4,\Sp_2\times \Sp_2,1,0)$ & P & W & $(\Sp_4,\SL_2,(2,1^2),std)$ \\
        \hline
        6 & $(\Sp_4,\Sp_2\times \Sp_2,1,T(std_{\SL_2}))$ & P & P & $(\Sp_4,\GL_1\times \Sp_2,1,0)$ \\
        \hline
        7 & $(\Sp_4,\Sp_2\times \Sp_2,1,T(std_{\SL_2,1}\oplus std_{\SL_2,2}))$ & P & PV & $(\Sp_4,\Sp_4,1,T(std_{\Sp_4}\oplus std_{\Sp_4}))$ \\
        \hline
        8 & $(\Sp_4,\Sp_2\times \Sp_2,1,T(std_{\SL_2,2}\oplus std_{\SL_2,2}))$ & P & P & self-dual \\
        \hline
        9 & $(\Sp_4,1,(4),0)$ & W & PV & $(\Sp_4,\Sp_4,1,0)$ \\
        \hline
        10 & $(\Sp_4,\GL_1,(2^2),0)$ & W & PV & $(\Sp_4,\Sp_4,1,T(std_{\Sp_4}))$ \\
        \hline
        11 & $(\Sp_4,\SL_2,(2,1^2),std)$ & W & P & $(\Sp_4,\Sp_2\times \Sp_2,1,0)$ \\
        \hline
	\end{tabular}
	\captionof{table}{Type $B_2=C_2$}
    \label{Table B_2 2}
\end{figure}

\begin{figure}[h!]
	\begin{tabular}{| c | c | c | c | c |}
		\hline
		\textnumero & $\Delta=(G,H,\iota,\rho_H)$ & \multicolumn{2}{c|}{Type} & $\hat{\Delta}=(\hat{G},\hat{H}',\hat{\iota}',\rho_{\hat{H}'})$  \\
		\hline
		1 & $(G_2,G_2,1,0)$ & PV & W & $(G_2,1,G_2,0)$ \\
		\hline
        2 & $(G_2,\SL_3,1,0)$ & P & W & $(G_2, \SL_2, A_1, std)$ \\
        \hline 
        3 & $(G_2, \SL_2, A_1, std)$ & W & P & $(G_2,\SL_3,1,0)$ \\
        \hline
        4 & $(G_2,1,G_2,0)$ & W & PV & $(G_2,G_2,1,0)$ \\
       \hline
	\end{tabular}
	\captionof{table}{Type $G_2$}
    \label{Table G_2 2}
\end{figure}

\begin{figure}[h!]
	\begin{tabular}{| c | c | c | c | c |}
		\hline
		\textnumero & $\Delta=(G,H,\iota,\rho_H)$ & \multicolumn{2}{c|}{Type} & $\hat{\Delta}=(\hat{G},\hat{H}',\hat{\iota}',\rho_{\hat{H}'})$  \\
		\hline
		1 & $(F_4,F_4,1,0)$ & PV & W & $(F_4,1,F_4,0)$ \\
		\hline
        2 & $(F_4,\Spin_9,1,0)$ & P & W & $(F_4,\SL_2,C_3,0)$ \\
        \hline
        3 & $(F_4,\Sp_6,A_1,std)$ & W & W & self-dual \\
        \hline 
        4 & $(F_4,\Spin_6,A_1,0)$ & W & W & self-dual \\
        \hline
        5 & $(F_4,G_2,A_2,0)$ & W & W & self-dual \\
        \hline
        6 & $(F_4,\SL_2,C_3,0)$ & W & P & $(F_4,\Spin_9,1,0)$ \\
        \hline
        7 & $(F_4,1,F_4,0)$ & W & PV & $(F_4,F_4,1,0)$ \\
        \hline
	\end{tabular}
	\captionof{table}{Type $F_4$}
    \label{Table F_4}
\end{figure}

\clearpage

\begin{figure}[h!]
	\begin{tabular}{| c | c | c | c | c |}
		\hline
		\textnumero & $\Delta=(G,H,\iota,\rho_H)$ & \multicolumn{2}{c|}{Type} & $\hat{\Delta}=(\hat{G},\hat{H}',\hat{\iota}',\rho_{\hat{H}'})$  \\
		\hline
		1 & $(E_6,E_6,1,0)$ & PV & W & $(E_6,1,E_6,0)$ \\
		\hline
        2 & $(E_6,E_6,1,T(std))$ & PV & W & $(E_6,\SL_3,D_4,T(std))$ \\
        \hline
        3 & $(E_6,F_4,1,0)$ & P & W & $(E_6,\SL_3,D_4,0)$ \\
        \hline
        4 & $(E_6,D_5,1,0)$ & P & W & $(E_6, \Sp_4,A_3,0)$ \\
        \hline
        5 & $(E_6,A_5,A_1,0)$ & W & W & $(E_6,G_2,A_2\times A_2,0)$ \\
        \hline
        6 & $(E_6,A_5,A_1,T(std))$ & W & W & $(E_6,\SO_7, A_1\times A_1,0)$ \\
        \hline
        7 & $(E_6,\SO_7, A_1\times A_1,0)$ & W & W & $(E_6,A_5,A_1,T(std))$ \\
        \hline
        8 & $(E_6,G_2,A_2\times A_2,0)$ & W & W & $(E_6,A_5,A_1,0)$ \\
        \hline
        9 & $(E_6, \Sp_4,A_3,0)$ & W & P & $(E_6,D_5,1,0)$ \\
        \hline
        10 & $(E_6,\SL_3,D_4,0)$ & W & P & $(E_6,F_4,1,0)$ \\
        \hline
        11 & $(E_6,\SL_3,D_4,T(std))$ & W & PV & $(E_6,E_6,1,T(std))$ \\
        \hline
        12 & $(E_6,1,E_6,0)$ & W & PV & $(E_6,E_6,1,0)$ \\
        \hline
	\end{tabular}
	\captionof{table}{Type $E_6$}
    \label{Table E_6}
\end{figure}

\begin{figure}[h!]
    \begin{tabular}{| c | c | c | c | c |}
        \hline
        \textnumero & $\Delta=(G,H,\iota,\rho_H)$ & \multicolumn{2}{c|}{Type} & $\hat{\Delta}=  (\hat{G},\hat{H}',\hat{\iota}',\rho_{\hat{H}'})$  \\
        \hline
        1 & $(E_7,E_7,1,0)$ & PV & W & $(E_7,1,E_7,0)$ \\
        \hline
        2 & $(E_7,E_7,1,std)$ & V & W & $(E_7,\PGL_2, E_6,0)$ \\
        \hline
        3 & $(E_7,E_6,1,0)$ & P & W & $(E_7,\Sp_6,D_4,T(std))$ \\
        \hline
        4 & $(E_7,D_6,A_1,0)$ & W & W & $(E_7,G_2, A_5,0)$ \\
        \hline
        5 & $(E_7,F_4,A_1\times A_1\times A_1,0)$ & W & W & $(E_7,\Sp_6,D_4,0)$ \\
        \hline
        6 & $(E_7,\Sp_6,D_4,0)$ & W & W & $(E_7,F_4,A_1\times A_1\times A_1,0)$ \\
        \hline
        7 & $(E_7,\Sp_6,D_4,T(std))$ & W & P & $(E_7,E_6,1,0)$ \\
        \hline
        8 & $(E_7,G_2, A_5,0)$ & W & W & $(E_7,D_6,A_1,0)$ \\
        \hline
        9 & $(E_7,\PGL_2, E_6,0)$ & W & V & $(E_7,E_7,1,std)$ \\
        \hline
        10 & $(E_7,1,E_7,0)$ & W & PV & $(E_7,E_7,1,0)$ \\
        \hline
    \end{tabular}
\captionof{table}{Type $E_7$}
\label{Table E_7}
\end{figure}

\begin{figure}[h!]
	\begin{tabular}{| c | c | c | c | c |}
		\hline
		\textnumero & $\Delta=(G,H,\iota,\rho_H)$ & \multicolumn{2}{c|}{Type} & $\hat{\Delta}=(\hat{G},\hat{H}',\hat{\iota}',\rho_{\hat{H}'})$  \\
		\hline
		1 & $(E_8,E_8,1,0)$ & PV & W & $(E_8,1,E_8,0)$ \\
		\hline
        2 & $(E_8,E_7, A_1,0)$ & W & W & $(E_8, G_2,E_6,0)$ \\
        \hline
        3 & $(E_8,F_4,D_4,0)$ & W & W & self-dual \\
        \hline
        4 & $(E_8, G_2,E_6,0)$ & W & W & $(E_8,E_7, A_1,0)$ \\
        \hline
        5 & $(E_8,1,E_8,0)$ & W & PV & $(E_8,E_8,1,0)$ \\
        \hline
	\end{tabular}
	\captionof{table}{Type $E_8$}
    \label{Table E8}
\end{figure}

\begin{figure}[h!]\leftskip-2cm
	\begin{tabular}{| c | c | c | c | c |}
		\hline
		\textnumero & $\Delta=(G,H,\iota,\rho_H)$ & \multicolumn{2}{c|}{Type} & $\hat{\Delta}=(\hat{G},\hat{H}',\hat{\iota}',\rho_{\hat{H}'})$  \\
		\hline
		1 & $(\GL_n,\GL_n,1,0)$ & PV & W & $(\GL_n,1,(n),0)$ \\
		\hline
        2 & $(\GL_n,\GL_n,1,T(std))$ & PV & W & $(\GL_n,\GL_1,(n-1,1),0)$ \\
        \hline
        3 & $(\GL_{2n},\GL_{2n},1,T(\wedge^2))$ & PV & W & $(\GL_{2n},\GL_n,(2^n),T(std))$ \\
        \hline
        4 & $(\GL_{2n+1},\GL_{2n+1},1,T(\wedge^2))$ & PV & W & $(\GL_{2n+1},\GL_n,(2^n,1),0)$ \\
        \hline
        5 & $(\GL_{2n},\GL_{2n},1,T(\wedge^2\oplus std))$ & PV & P & $(\GL_{2n},\GL_n\times \GL_n,1,T(std))$ \\
        \hline
        6 & $(\GL_{2n+1},\GL_{2n+1},1,T(\wedge^2\oplus std))$ & PV & P & $(\GL_{2n+1},\GL_{n+1}\times \GL_n,1,T(std_{\GL_{n+1}}))$ \\
        \hline
        7 & $(\GL_n,\GL_n,1,T(std\oplus std))$ & PV & W & $(\GL_n,\GL_2,(n-2,1^2),T(std))$ if $n>2$; self-dual if $n=2$ \\
        \hline
        8 & $(\GL_6,\GL_6,1,\wedge^3)$ & V & W & $(\GL_6,\GL_2,(3^2),0)$ \\
        \hline
        9 & $(\GL_6,\GL_6,1,\wedge^3\oplus T(std))$ & V & W & $(\GL_6,\GL_2\times \GL_2,(2^21^2),0)$ \\
        \hline
        10 & $(\GL_{2n+1},\Sp_{2n},1,0)$ & P & P & $(\GL_{2n+1},\GL_{n+1}\times \GL_n,1,0)$ \\
        \hline
        11 & $(\GL_{2n},\Sp_{2n},1,0)$ & P & W & $(\GL_{2n},\GL_n,(2^n),0)$ \\
        \hline
        12 & $(\GL_{2n},\Sp_{2n},1,T(std))$ & P & P & $(\GL_{2n},\GL_n\times \GL_n,1,0)$ \\
        \hline
        13 & $(\GL_{2n},\GL_n\times \GL_n,1,0)$ & P & P & $(\GL_{2n},\Sp_{2n},1,T(std))$ \\
        \hline
        14 & $(\GL_{a+b},\GL_a\times \GL_b,1,0),\;a<b$ & P & W & $(\GL_{a+b},\Sp_{2a},(b-a,1^{2a}),0)$ \\
        \hline
        15 & $(\GL_{2n},\GL_n\times \GL_n,1,T(std))$ & P & PV & $(\GL_{2n},\GL_{2n},1,T(\wedge^2\oplus std))$ \\
        \hline
        16 & $(\GL_{a+b},\GL_a\times \GL_b,1,T(std_{\GL_a})),\;a<b$ & P & W & $(\GL_{a+b},\GL_{2a},(b-a,1^{2a}),T(\wedge^2))$ \\
        \hline
        17 & $(\GL_{2n+1},\GL_{n+1}\times \GL_n,1,T(std_{\GL_{n+1}}))$ & P & PV & $(\GL_{2n+1},\GL_{2n+1},1,T(\wedge^2\oplus std))$ \\
        \hline
        18 & $(\GL_{a+b},\GL_a\times \GL_b,1,T(std_{\GL_b})),\;a+1<b$ & P & W & $(\GL_{a+b},\GL_{2a+1},(b-a-1,1^{2a+1}),T(\wedge^2))$ \\
        \hline
        19 & $(\GL_{2n+1},\GL_{2n},1,T(\wedge^2))$ & P & P & $(\GL_{2n+1},\GL_{n+1}\times \GL_n,1,T(std_{\GL_{n}}))$ \\
        \hline
        20 & $(\GL_{2n},\GL_{2n-1},1,T(\wedge^2))$ & P & P & $(\GL_{2n},\GL_{n+1}\times \GL_{n-1},1,T(std_{\GL_{n+1}}))$ \\
        \hline
        21 & $(\GL_8,\GL_6\times \GL_2,1,\wedge^3)$ & ? & ? & self-dual \\
        \hline
        22 & $(\GL_7,\GL_6\times \GL_1,1,\wedge^3)$ & ? & W & $(\GL_7,\GL_2\times \GL_3,(2^21^3),0)$ \\
        \hline
        23 & $(\GL_n,1,(n),0)$ & W & PV & $(\GL_n,\GL_n,1,0)$ \\
        \hline
        24 & $(\GL_{2n},\GL_n,(2^n),0)$ & W & P & $(\GL_{2n},\Sp_{2n},1,0)$ \\
        \hline
        25 & $(\GL_{2n},\GL_n,(2^n),T(std))$ & W & PV & $(\GL_{2n},\GL_{2n},1,T(\wedge^2))$ \\
        \hline
        26 & $(\GL_{2n+1},\GL_n,(2^n,1),0)$ & W & PV & $(\GL_{2n+1},\GL_{2n+1},1,T(\wedge^2))$ \\
        \hline
        27 & $(\GL_6,\GL_2,(3^2),0)$ & W & V & $(\GL_6,\GL_6,1,\wedge^3)$ \\
        \hline
        28 & $(\GL_{n+4},\GL_2\times \GL_{n},(2^2 1^n),0),n\geq 3$ & W & W & $(\GL_{n+4},\GL_6,(n-2,1^6),\wedge^3)$ \\
        \hline
        29 & $(\GL_6,\GL_2\times \GL_2,(2^21^2),0)$ & W & V & $(\GL_6,\GL_6,1,\wedge^3\oplus T(std))$ \\
        \hline
        30 & $(\GL_{n+2k},\Sp_{2k},(n,1^{2k}),0)$ & W & P & $(\GL_{n+2k},\GL_{n+k}\times \GL_k,1,0)$ \\
        \hline
        31 & $(\GL_{n+k},\GL_k,(n,1^{k}),0),k>1$ & W & W & $(\GL_{n+k},\GL_{n+1},(k-1,1^{n+1}),0)$ \\
        \hline
        32 & $(\GL_{n+1},\GL_1,(n,1),0)$ & W & PV & $(\GL_{n+1},\GL_{n+1},1,T(std))$ \\
        \hline
        33 & $(\GL_{n+2},\GL_2,(n,1^2),T(std))$ & W & PV & $(\GL_{n+2},\GL_{n+2},1,T(std\oplus std))$ \\
        \hline
        34 & $(\GL_{n+k},\GL_k,(n,1^{k}),T(std)),k>2$ & W & W & $(\GL_{n+k},\GL_{n+2},(k-2,1^{n+2}),T(std))$ \\
        \hline
        35 & $(\GL_{n+2k},\GL_{2k},(n,1^{2k}),T(\wedge^2)),k>1$ & W & P & $(\GL_{n+2k},\GL_{n+k}\times \GL_k,1,T(std_{\GL_k}))$ \\
        \hline
        36 & $(\GL_{n+2k+1},\GL_{2k+1},(n,1^{2k+1}),T(\wedge^2)),k>1$ & W & P & $(\GL_{n+2k+1},\GL_{n+k+1}\times \GL_k,1,T(std_{\GL_{n+k+1}}))$ \\
        \hline
        37 & $(\GL_{n+6},\GL_6,(n,1^{6}),\wedge^3)$ & ? & W & $(\GL_{n+6}, \GL_{n+2}\times \GL_2,(2^2 1^{n+2}),0)$ \\
        \hline
	\end{tabular}
	\captionof{table}{Type A}
    \label{Table A}
\end{figure}

\newpage

\begin{figure}[h!]\leftskip-2cm
	\begin{tabular}{| c | c | c | c | c |}
		\hline
		\textnumero & $\Delta=(G,H,\iota,\rho_H)$ & \multicolumn{2}{c|}{Type} & $\hat{\Delta}=(\hat{G},\hat{H}',\hat{\iota}',\rho_{\hat{H}'})$  \\
		\hline
		1 & $(\SO_{2n+1},\SO_{2n+1},1,0)$ & PV & W & $(\Sp_{2n},1,(2n),0)$ \\
		\hline
        2 & $(\SO_{11},\SO_{11},1,\Spin_{11})$ & V & W & $(\Sp_{10},\SL_2, (5^2),0)$ \\
        \hline
        3 & $(\SO_{13},\SO_{13},1,\Spin_{13})$ & V & W & $(\Sp_{12},\Sp_4, (3^4),0)$ \\
        \hline
        4 & $(\SO_{7},\SO_{7},1,T(\Spin_{7}))$ & PV & W & $(\Sp_6,\SL_2,(3^2),T(std))$ \\
        \hline
        5 & $(\SO_{9},\SO_{9},1,T(\Spin_{9}))$ & PV & W & $(\Sp_8,\SL_2\times \SL_2,(3^2 1^2),T(std_{\SL_2}))$ \\
        \hline
        6 & $(\SO_{4n+1},\GL_{2n},1,0)$ & P & P & $(\Sp_{4n},\Sp_{2n}\times \Sp_{2n},1,T(std_{\Sp_{2n}}))$ \\
        \hline
        7 & $(\SO_{4n+3},\GL_{2n+1},1,0)$ & P & P & $(\Sp_{4n+2},\Sp_{2n+2}\times \Sp_{2n},1,T(std_{\Sp_{2n+2}}))$ \\
        \hline
        8 & $(\SO_7,G_2,1,0)$ & P & W & $(\Sp_6,\SL_2,(3^2),0)$ \\
        \hline
        9 & $(\SO_9,\Spin_7,1,0)$ & P & W & $(\Sp_8,\SL_2\times \SL_2,(3^2 1^2),0)$ \\
        \hline
        10 & $(\SO_{2n+1},\SO_{2n},1,0)$ & P & W & $(\Sp_{2n},\SL_2,(2n-2, 1^2),std))$ \\
        \hline
        11 & $(\SO_{11},\SO_{10},1,T(HSpin_{10}))$ & P & W & $(\Sp_{10},\Sp_2\times \Sp_4,(3^21^4),T(std_{\Sp_4}))$ \\
        \hline
        12 & $(\SO_{13},\SO_{12},1,HSpin_{12})$ & ? & W & $(\Sp_{12},\SL_2\times \SL_2,(5^2 1^2),0)$ \\
        \hline
        13 & $(\SO_{9},\SO_{8},1,T(HSpin_{8}))$ & P & W & $(\Sp_8,\SL_2\times \SL_2,(3^2 1^2),T(std_{\SL_2}))$ \\
        \hline
        14 & $(\SO_{7},\SO_{6},1,T(HSpin_{6}))$ & P & W & $(\Sp_6,\GL_1\times \SL_2,(2^2 1^2),0)$ \\
        \hline
        15 & $(\SO_{7},\SO_{6},1,T(HSpin_{6}\oplus HSpin_6)$ & P & P & $(\Sp_6,\Sp_4\times \Sp_2,1,T(std_{\Sp_4}\oplus std_{\Sp_2}))$ \\
        \hline
        16 & $(\SO_{2n+1},1,(2n+1),0)$ & W & PV & $(\Sp_{2n},\Sp_{2n},1,0)$ \\
        \hline
        17 & $(\SO_{4n+1},\Sp_{2n},(2^{2n} 1),std)$ & W & P & $(\Sp_{4n},\Sp_{2n}\times \Sp_{2n},1,0)$ \\
        \hline
        18 & $(\SO_{4n+2m+1},\Sp_{2n},(2^{2n}, 2m+1),0),\;m>0$ & W & P & $(\Sp_{4n+2m},\Sp_{2n}\times \Sp_{2n+2m},1,0)$ \\
        \hline
        19 & $(\SO_{2m+3},\SO_{2},(2m+1,1^{2}),0)$ & W & PV & $(\Sp_{2m+2},\Sp_{2m+2},1,T(std))$ \\
        \hline
        20 & $(\SO_{2n+2m+1},\SO_{2n},(2m+1,1^{2n}),0),\;n>1$ & W & W & $(\Sp_{2n+2m},\Sp_{2m+2},(2n-2,1^{2m+2}),std)$ \\
        \hline
        21 & $(\SO_{2n+11},\SO_{10},(2n+1,1^{10}),T(HSpin_{10}))$ & W & W & $(\Sp_{2n+10},\Sp_2\times \Sp_{2n+4},(3^2 1^{2n+4}),T(std_{\Sp_{2n+4}}))$ \\
        \hline
        22 & $(\SO_{2n+9},\SO_{8},(2n+1,1^{8}),T(HSpin_{8}))$ & W & W & $(\Sp_{2n+8},\Sp_2\times \Sp_{2n+2},(3^2 1^{2n+2}),T(std_{\Sp_{2}}))$ \\
        \hline
        23 & $(\SO_{2n+13},\SO_{12},(2n+1,1^{12}),HSpin_{12})$ & ? & W & $(\Sp_{2n+12},\Sp_2\times \Sp_{2n+2},(5^2 1^{2n+2}),0)$ \\
        \hline
        24 & $(\SO_{2n+5},\SO_{4},(2n+1,1^{4}),T(HSpin_{4}))$ & W & P & $(\Sp_{2n+4},\GL_1\times \Sp_{2n+2},1,0)$ \\
        \hline
        25 & $(\SO_{2n+5},\SO_{4},(2n+1,1^{4}),T(HSpin_{4}^{+}\oplus HSpin_{4}^{-}))$ & W & PV & $(\Sp_{2n+4},\Sp_{2n+4},1,T(std\oplus std))$ \\
        \hline
        26 & $(\SO_{2n+5},\SO_{4},(2n+1,1^{4}),T(HSpin_{4}^{+}\oplus HSpin_{4}^{+}))$ & W & P & $(\Sp_{2n+4},\Sp_{2n+2}\times \Sp_2,1,T(std_{\Sp_{2}\oplus std_{\Sp_2}}))$ \\
        \hline
        27 & $(\SO_{2n+7},\SO_{6},(2n+1,1^{6}),T(HSpin_{6}))$ & W & W & $(\Sp_{2n+6},\GL_1\times \Sp_{2n+2},(2^2 1^{2n+2}),0)$ \\
        \hline
        28 & $(\SO_{2n+7},\SO_{6},(2n+1,1^{6}),T(HSpin_{6}\oplus HSpin_6))$ & W & P & $(\Sp_{2n+6},\Sp_{2n+4}\times \Sp_2,1,T(std_{\Sp_{2n+4}\oplus std_{\Sp_2}}))$ \\
        \hline
        29 & $(\SO_{4n+2m+1},\GL_{2n},(2m+1,1^{4n}),0)$ & W & P & $(\Sp_{4n+2m},\Sp_{2n+2m}\times \Sp_{2n},1,T(std_{\Sp_{2n}}))$ \\
        \hline
        30 & $(\SO_{4n+2m+3},\GL_{2n+1},(2m+1,1^{4n+2}),0)$ & W & P & $(\Sp_{4n+2m+2},\Sp_{2n+2m+2}\times \Sp_{2n},1,T(std_{\Sp_{2n+2m+2}}))$ \\
        \hline
        31 & $(\SO_{2n+9},\Spin_7,(2n+1,1^8),0)$ & W & W & $(\Sp_{2n+8},\SL_2\times \Sp_{2n+2},(3^2 1^{2n+2}),0)$ \\
        \hline
        32 & $(\SO_7,\SL_2\times \SO_3,(2^2 1^3),std_{\SL_2})$ & W & V & $(\Sp_6,\Sp_6,1,\wedge^3\oplus std)$ \\
        \hline
        33 & $(\SO_{2n+5},\SL_2\times \SO_{2n+1}, (2^2 1^{2n+1}),std_{\SL_2}),\;n>1$ & W & W & $(\Sp_{2n+4},\Sp_6,(2n-2,1^6),\wedge^3)$ \\
        \hline
	\end{tabular}
	\captionof{table}{Type $B_n$ for $n>2$}
    \label{Table B}
\end{figure}

\newpage

\begin{figure}[h!]\leftskip-2cm
	\begin{tabular}{| c | c | c | c | c |}
		\hline
		\textnumero & $\Delta=(G,H,\iota,\rho_H)$ & \multicolumn{2}{c|}{Type} & $\hat{\Delta}=(\hat{G},\hat{H}',\hat{\iota}',\rho_{\hat{H}'})$  \\
		\hline
		1 & $(\Sp_{2n},\Sp_{2n},1,0)$ & PV & W & $(\SO_{2n+1},1,(2n+1),0)$ \\
		\hline
        2 & $(\Sp_{2n},\Sp_{2n},1,T(std))$ & PV & W & $(\SO_{2n+1},\SO_2,(2n-1,1^2),0)$ \\
        \hline
        3 & $(\Sp_{2n},\Sp_{2n},1,T(std\oplus std))$ & PV & W & $(\SO_{2n+1},\SO_4,(2n-3,1^4),T(HSpin_{4}^{+}\oplus HSpin_{4}^{-}))$ \\
        \hline
        4 & $(\Sp_{6},\Sp_{6},1,\wedge^3\oplus std)$ & V & W & $(\SO_7,\Sp_2\times \SO_3,(2^21^3),std_{\Sp_2})$ \\
        \hline
        5 & $(\Sp_{2n},\GL_1\times \Sp_{2n-2},1,0)$ & P & W & $(\SO_{2n+1},\SO_4,(2n-3,1^4),T(HSpin_4))$ \\
        \hline
        6 & $(\Sp_{4n},\Sp_{2n}\times \Sp_{2n},1,0)$ & P & W & $(\SO_{4n+1},\Sp_{2n},(2^{2n}1),std)$ \\
        \hline
        7 & $(\Sp_{4n},\Sp_{2n}\times \Sp_{2n},1,T(std))$ & P & P & $(\SO_{4n+1},\GL_{2n},1,0)$ \\
        \hline
        8 & $(\Sp_{2n},\Sp_{2m}\times \Sp_{2n-2m},1,0),\;n>2m$ & P & W & $(\SO_{2n+1},\Sp_{2m},(2^{2m},2n-2m+1),0)$ \\
        \hline
        9 & $(\Sp_{2n},\Sp_{2m}\times \Sp_{2n-2m},1,T(std_{\Sp_{2m}})),\;n>2m$ & P & W & $(\SO_{2n+1},\GL_{2m},(2n-4m+1,1^{4m}),0)$ \\
        \hline
        10 & $(\Sp_{2n},\Sp_{2m}\times \Sp_{2n-2m},1,T(std_{2n-2m})),\;n>2m$ & P & W & $(\SO_{2n+1},\GL_{2m+1},(2n-4m-1,1^{4m+2}),0)$ \\
        \hline
        11 & $(\Sp_{2n},\Sp_{2}\times \Sp_{2n-2},1,T(std_{\Sp_2}\oplus std_{\Sp_2}))$ & P & W & $(\SO_{2n+1},\SO_4,(2n-3,1^4),T(HSpin_{4}^{+}\oplus HSpin_{4}^{+}))$ \\
        \hline
        12 & $(\Sp_{2n},\Sp_{2}\times \Sp_{2n-2},1,T(std_{\Sp_2}\oplus std_{\Sp_{2n-2}}))$ & P & W & $(\SO_{2n+1},\SO_6,(2n-5,1^6),T(HSpin_6\oplus HSpin_6))$ \\
        \hline
        13 & $(\Sp_{2n},1,(2n),0)$ & W & PV & $(\SO_{2n+1},\SO_{2n+1},1,0)$ \\
        \hline
        14 & $(\Sp_{12},\Sp_4, (3^4),0)$ & W & V & $(\SO_{13},\SO_{13},1,\Spin_{13})$ \\
        \hline
        15 & $(\Sp_{2n+4},\GL_1\times \Sp_{2n},(2^21^{2n}),0)$ & W & W & $(\SO_{2n+5},\SO_6,(2n-1,1^6),T(HSpin_6))$ \\
        \hline
        16 & $(\Sp_{10},\SL_2, (5^2),0)$ & W & V & $(\SO_{11},\SO_{11},1,\Spin_{11})$ \\
        \hline
        17 & $(\Sp_{2n+10},\Sp_2\times \Sp_{2n},(5^2 1^{2n}),0)$ & W & ? & $(\SO_{2n+11},\SO_{12}, (2n-1,1^{12}),HSpin_{12})$ \\
        \hline
        18 & $(\Sp_6,\SL_2,(3^2),0)$ & W & P & $(\SO_7,G_2,1,0)$ \\
        \hline
        19 & $(\Sp_6,\SL_2,(3^2),T(std))$ & W & PV & $(\SO_{7},\SO_{7},1,T(\Spin_{7}))$ \\
        \hline
        20 & $(\Sp_{2n+6},\Sp_2\times \Sp_{2n},(3^2 1^{2n}),0)$ & W & W & $(\SO_{2n+7},\Spin_7,(2n-1,1^8),0)$ \\
        \hline
        21 & $(\Sp_{2n+6},\Sp_2\times \Sp_{2n},(3^2 1^{2n}),T(std_{\Sp_2}))$ & W & W & $(\SO_{2n+7},\SO_8,(2n-1,1^8),T(HSpin_8))$ \\
        \hline
        22 & $(\Sp_{2n+6},\Sp_2\times \Sp_{2n},(3^2 1^{2n}),T(std_{\Sp_{2n}})),\;n>1$ & W & W & $(\SO_{2n+7},\SO_{10},(2n-3,1^{10}),T(HSpin_{10}))$ \\
        \hline
        23 & $(\Sp_8,\SL_2\times \SL_2,(3^2 1^2),T(std_{\SL_2}))$ & W & PV & $(\SO_{9},\SO_{9},1,T(\Spin_{9}))$ \\
        \hline
        24 & $(\Sp_{2n+2m},\Sp_{2n},(2m,1^{2n}),std)$ & W & W & $(\SO_{2n+2m+1},\SO_{2m+2},(2n-1,1^{2m+2}),0)$ \\
        \hline
        25 & $(\Sp_{2n+6},\Sp_6,(2n,1^6),\wedge^3)$ & W & W & $(\SO_{2n+7},\Sp_2\times \SO_{2n+3}, (2^21^{2n+3}),std_{\Sp_2})$ \\
        \hline
	\end{tabular}
	\captionof{table}{Type $C_n$ for $n>2$}
    \label{Table C}
\end{figure}

\begin{figure}[h!]
	\begin{tabular}{| c | c | c | c | c |}
		\hline
		\textnumero & $\Delta=(G,H,\iota,\rho_H)$ & \multicolumn{2}{c|}{Type} & $\hat{\Delta}=(\hat{G},\hat{H}',\hat{\iota}',\rho_{\hat{H}'})$  \\
		\hline
		1 & $(\SO_{2n},\SO_{2n},1,0)$ & PV & W & $(\SO_{2n},1,(2n-1, 1),0)$ \\
		\hline
        2 & $(\SO_{2n},\SO_{2n},1,T(std))$ & PV & W & $(\SO_{2n},\SO_3,(2n-3,1^{3}),T(\Spin_3))$ \\
        \hline
        3 & $(\SO_{10},\SO_{10},1,T(HSpin_{10}))$ & PV & W & $(\SO_{10},\SL_2,(4^2,1^2),0)$ \\
        \hline
        4 & $(\SO_{12},\SO_{12},1,HSpin_{12})$ & V & W & $(\SO_{12},\SL_2,(6^2),0)$ \\
        \hline
        5 & $(\SO_{12},\SO_{12},1,HSpin_{12}^{+}\oplus HSpin_{12}^{-})$ & V & ? & $(\SO_{12},\Sp_4\times \SO_4,(2^41^4),0)$ \\
        \hline
        6 & $(\SO_{12},\SO_{12},1,HSpin_{12}\oplus T(std))$ & V & ? & $(\SO_{12},\SL_2\times \SO_4,(4^21^4),T(HSpin_4))$ \\
        \hline
        7 & $(\SO_{8},\SO_{8},1,T(std\oplus HSpin_8))$ & PV & W & $(\SO_8,\SL_2\times \SO_4,(2^2,1^4),T(HSpin_{4}^{+}\oplus HSpin_{4}^{-}))$ \\
        \hline
        8 & $(\SO_8,G_2,1,0)$ & P & W & $(\SO_8,\SL_2\times \SO_4,(2^2,1^4),0)$ \\
        \hline
        9 & $(\SO_{10},\Spin_7\times \GL_1,1,0)$ & P & W & $(\SO_{10},\SL_2\times \SO_6,(2^21^6),T(HSpin_6))$ \\
        \hline
        10 & $(\SO_{4n},\GL_{2n},1,0)$ & P & W & $(\SO_{4n},\Sp_{2n},(2^{2n}),T(std))$ \\
        \hline
\end{tabular}
\end{figure}

\newpage

\begin{figure}[h!]\leftskip-2cm
	\begin{tabular}{| c | c | c | c | c |}
		\hline
		\textnumero & $\Delta=(G,H,\iota,\rho_H)$ & \multicolumn{2}{c|}{Type} & $\hat{\Delta}=(\hat{G},\hat{H}',\hat{\iota}',\rho_{\hat{H}'})$  \\
		\hline
        11 & $(\SO_{4n+2},\GL_{2n+1},1,0)$ & P & W & $(\SO_{4n+2},\Sp_{2n}\times \SO_2,(2^{2n}1^2),0)$ \\
        \hline
        12 & $(\SO_{2n},\GL_{n},1,T(std))$ & P & P & self-dual \\
        \hline
        13 & $(\SO_{2n},\SO_{2n-1},1,0)$ & P & W & $(\SO_{2n},\SO_3,(2n-3,1^3),0)$ \\
        \hline
        14 & $(\SO_{2n},\SO_{2n-2}\times \SO_2,1,0)$ & P & W & $(\SO_{2n},\SO_5,(2n-5,1^5),T(\Spin_5))$ \\
        \hline
        15 & $(\SO_{14},\SO_{13},1,\Spin_{13})$ & ? & $\ast$ & $(\SO_{14},\Sp_4\times \SO_6,(2^41^6),0)$ \\
        \hline
        16 & $(\SO_{12},\SO_{11},1,\Spin_{11})$ & ? & $\ast$ & $(\SO_{12},\SL_2\times \SO_4,(4^21^4),0)$ \\
        \hline
        17 & $(\SO_{8},\SO_{7},1,T(\Spin_{7}))$ & P & W & $(\SO_8,\Sp_2\times \SO_4,(2^21^4),T(std_{\Sp_2}))$ \\
        \hline
        18 & $(\SO_{14},\SO_{12}\times \SO_2,1,HSpin_{12})$ & ? & $\ast$ & $(\SO_{14},\SL_2\times \SO_6,(4^21^6),T(HSpin_6))$ \\
        \hline
        19 & $(\SO_{10},\SO_{8}\times \SO_2,1,T(HSpin_{8}))$ & P & W & $(\SO_{10},\SO_6\times \Sp_2,(2^21^6),T(HSpin_6\oplus std_{\Sp_2}))$ \\
        \hline
        20 & $(\SO_{2n},1,(2n-1, 1),0)$ & W & PV & $(\SO_{2n},\SO_{2n},1,0)$ \\
        \hline
        21 & $(\SO_{4n+2},\Sp_{2n}\times \SO_2,(2^{2n}1^2),0)$ & W & P & $(\SO_{4n+2},\GL_{2n+1},1,0)$ \\
        \hline
        22 & $(\SO_{2n+4},\Sp_2\times \SO_{2n},(2^21^{2n}),0)$ & W & W & $(\SO_{2n+4},G_2,(2k-3,1^7),0)$ \\
        \hline
        23 & $(\SO_{2n+4},\Sp_2\times \SO_{2n},(2^21^{2n}),T(std_{\Sp_2}))$ & W & W & $(\SO_{2n+4},\Spin_7,(2k-3,1^7),T(\Spin_7))$ \\
        \hline
        24 & $(\SO_8,\SL_2\times \SO_4,(2^2,1^4),T(HSpin_{4}^{+}\oplus HSpin_{4}^{-}))$ & W & PV & $(\SO_{8},\SO_{8},1,T(std\oplus HSpin_8))$ \\
        \hline
        25 & $(\SO_{10},\SL_2\times \SO_6,(2^21^6),T(HSpin_6))$ & W & P & $(\SO_{10},\Spin_7\times \GL_1,1,0)$ \\
        \hline
        26 & $(\SO_{10},\SO_6\times \Sp_2,(2^21^6),T(HSpin_6\oplus std_{\Sp_2}))$ & W & P & $(\SO_{10},\SO_{8}\times \SO_2,1,T(HSpin_{8}))$ \\
        \hline
        27 & $(\SO_{12},\SL_2\times \SO_8,(2^21^8),T(HSpin_8))$ & W & ? & $(\SO_{12},\SL_2\times \Spin_7,(2^21^8),T(std_{\SL_2}))$ \\
        \hline
        28 & $(\SO_{12},\SL_2\times \SO_8,(2^21^8),T(HSpin_8\oplus std_{\SL_2}))$ & W & W & self-dual \\
        \hline
        29 & $(\SO_{14},\SL_2\times \SO_{10},(2^21^{10}),T(HSpin_{10}))$ & W & W & self-dual \\
        \hline
        30 & $(\SO_{16},\SL_2\times \SO_{12},(2^21^{12}),HSpin_{12})$ & W & ? & $(\SO_{16},\Sp_2\times \Spin_7,(4^21^4),0)$ \\
        \hline
        31 & $(\SO_{16},\SL_2\times \SO_{12},(2^21^{12}),T(HSpin_{12}\oplus std_{\SL_2}))$ & W & ? & $(\SO_{16},\Sp_2\times \SO_8,(4^21^4),T(HSpin_8))$ \\
        \hline
        32 & $(\SO_{12},\SL_2\times \Spin_7,(2^21^8),T(std_{\SL_2}))$ & ? & W & $(\SO_{12},\SL_2\times \SO_8,(2^21^8),T(HSpin_8))$ \\
        \hline
        33 & $(\SO_{12},\SL_2\times \Spin_7,(2^21^8),0)$ & $\ast$ & $\ast$ & self-dual \\
        \hline
        34 & $(\SO_{12},\Sp_4\times \SO_4,(2^41^4),0)$ & ? & V & $(\SO_{12},\SO_{12},1,HSpin_{12}^{+}\oplus HSpin_{12}^{-})$ \\
        \hline
        35 & $(\SO_{2n+8},\Sp_4\times \SO_{2n},(2^41^{2n}),0),\;n>2$ & $\ast$ & ? & $(\SO_{2n+8},\SO_{13},(2n-5,1^{13}),\Spin_{13})$ \\
        \hline
        36 & $(\SO_{4n},\Sp_{2n},(2^{2n}),0)$ & W & W & self-dual \\
        \hline
        37 & $(\SO_{4n},\Sp_{2n},(2^{2n}),T(std))$ & W & P & $(\SO_{4n},\GL_{2n},1,0)$ \\
        \hline
        38 & $(\SO_{12},\SL_2,(6^2),0)$ & W & V & $(\SO_{12},\SO_{12},1,HSpin_{12})$ \\
        \hline
        39 & $(\SO_{10},\SL_2,(4^2,1^2),0)$ & W & PV & $(\SO_{10},\SO_{10},1,T(HSpin_{10}))$ \\
        \hline
        40 & $(\SO_{2n+8},\SL_2\times \SO_{2n},(4^2,1^{2n}),0),\;n>1$ & $\ast$ & ? & $(\SO_{2n+8},\SO_{11}, (2n-3,1^{11}),\Spin_{11})$ \\
        \hline
41 & $(\SO_{12},\SL_2\times \SO_4,(4^21^4),T(HSpin_4))$ & ? & V & $(\SO_{12},\SO_{12},1,HSpin_{12}\oplus T(std))$\\
        \hline
        42 & $(\SO_{14},\SL_2\times \SO_6,(4^21^6),T(HSpin_6))$ & $\ast$ & ? & $(\SO_{14},\SO_{12}\times \SO_2,1,HSpin_{12})$ \\
        \hline
        43 & $(\SO_{16},\SL_2\times \SO_8,(4^21^8),T(HSpin_8))$ & ? & W & $(\SO_{16},\SL_2\times \SO_{12},(2^21^{12}),T(HSpin_{12}\oplus std_{\SL_2}))$ \\
        \hline
        44 & $(\SO_{20},\SL_2\times \SO_{12},(4^21^{12}),HSpin_{12})$ & $\ast$ & $\ast$ & self-dual \\
        \hline
        45 & $(\SO_{16},\Sp_2\times \Spin_7,(4^21^8),0)$ & ? & W & $(\SO_{16},\SL_2\times \SO_{12},(2^21^{12}),HSpin_{12})$ \\
        \hline
        46 & $(\SO_{2n+2k},\SO_{2n-1},(2k+1,1^{2n-1}),0)$ & W & W & $(\SO_{2n+2k},\SO_{2k+3},(2n-3,1^{2k+3}),0)$ \\
        \hline
        47 & $(\SO_{2n+12},\SO_{11},(2n+1,1^{11}),\Spin_{11})$ & ? & $\ast$ & $(\SO_{2n+12},\SL_2\times \SO_{2n+4},(4^2,1^{2n+4}),0)$ \\
        \hline
        48 & $(\SO_{2n+14},\SO_{13},(2n+1,1^{13}),\Spin_{13})$ & ? & $\ast$ & $(\SO_{2n+14},\SO_{2n+6}\times \Sp_4,(2^4,1^{2n+6}),0)$ \\
        \hline
        49 & $(\SO_{2n+8},\SO_{7},(2n+1,1^{7}),T(\Spin_{7}))$ & W & W & $(\SO_{2n+8},\SL_2\times \SO_{2n+4},(2^2,1^{2n+4}),T(std_{\Sp_2}))$ \\
        \hline
        50 & $(\SO_{2n+6},\SO_{5},(2n+1,1^{5}),T(\Spin_{5}))$ & W & P & $(\SO_{2n+6},\SO_2\times \SO_{2n+4},1,0)$ \\
        \hline
        51 & $(\SO_{2n+4},\SO_{3},(2n+1,1^{3}),T(\Spin_{3}))$ & W & PV & $(\SO_{2n+4},\SO_{2n+4},1,T(std))$ \\
        \hline
        52 & $(\SO_{2n+8},G_2,(2n+1,1^7),0)$ & W & W & $(\SO_{2n+8},\SL_2\times \SO_{2n+4},(2^21^{2n+4}),0)$ \\
        \hline
	\end{tabular}
	\captionof{table}{Type $D_n$, $n>3$}
    \label{Table D}
\end{figure}

\clearpage

To end this section, we discuss the cases with Type $\ast$ in Table \ref{Table D} above. For those cases, we need to use some other argument to compute the dual of $\Delta_{red}$ and then use Conjecture \ref{Whittaker induction} to compute the dual of $\Delta$. We start with the quadruple 
$$\Delta_{red}=(\GL_4\times \SO_{2n},\Sp_4\times \SO_{2n}, 1,std\otimes std),\;n\geq 3.$$
This is neither a polarized nor vector space case, but we can compute its dual using the period integral conjecture. In fact, the period integral associated with $\Delta_{red}$ is a combination of the theta correspondence between $\Sp_4\times \SO_{2n}$ and the Gan-Gross-Prasad period $\SO_6\times \SO_5/\SO_5$ (which is just $\GL_4\times \Sp_4/\Sp_4$ up to isogeny). Then from the period integral conjecture point of view, the dual should be given by 
$$\widehat{(\Delta_{red})}=(\GL_4\times \SO_{2n},\GL_4\times \SO_5, (2n-5,1^5),\wedge_{\GL_4}^2\otimes \Spin_5).$$
Combining this duality with Conjecture \ref{Whittaker induction}, we can compute the dual for Type $\ast$ in Model 15, Model 35 and Model 48.

Next, we consider the quadruple 
$$\Delta_{red}=(\GL_2\times \GL_2\times \SO_{2n},\SL_2\times \SO_{2n}, 1,std\otimes std),\;n\geq 2.$$
We can still use the period integral conjecture to compute its dual. In this case, the period integral associated with $\Delta_{red}$ is a combination of the theta correspondence between $\Sp_2\times \SO_{2n}$ and the trilinear period $\GL_{2}^{3}/\GL_2$. Then from the period integral conjecture point of view, the dual should be given by 
$$\widehat{(\Delta_{red})}=(\GL_2\times \GL_2\times \SO_{2n},\GL_2\times \GL_2\times \SO_{3}, (2n-3,1^3),std\otimes std\otimes \Spin_3).$$
Combining this duality with Conjecture \ref{Whittaker induction}, we can compute the dual for Type $\ast$ in Model 16, Model 40, and Model 47. The arguments for the above two models are similar to those in Tables 21, 23, 25 of \cite{MWZ2} and we will skip the details here.

Next, we consider the quadruple 
$$\Delta_{red}=(\GL_2\times \GL_2\times \SO_{6},\SL_2\times \SO_{6}, 1,std\otimes std\oplus T(\HSpin_6)).$$
We claim that it is dual to 
$$\widehat{(\Delta_{red})}=(\GL_2\times \GL_2\times \SO_{6}, \GL_2\times \GL_2\times \SO_{4}\times \SO_2,1,std_{\GL_2}\otimes std_{\GL_2}\otimes \HSpin_4).$$
To justify this, we study the period integral associated with $\widehat{(\Delta_{red})}$. It is easy to see that this is a combination of the theta correspondence of $\SO_4\times \Sp_2$ (this corresponds to the integral over $\GL_2\times \GL_2$ against the theta series associated with the symplectic representation $std_{\GL_2}\otimes std_{\GL_2}\otimes \HSpin_4$) and the period integral for the spherical variety $\GL_4\times \GL_2/\GL_2\times \GL_2$ (this corresponds to the integral over $\SO_4\times \SO_2$) studied in \cite{WZ}. Then from the period integral conjecture point of view (with the help of the unramified computation in \cite{WZ} and the theta correspondence), the dual of $\widehat{(\Delta_{red})}$ should be $\Delta_{red}$. Combining with Conjecture \ref{Whittaker induction}, we can compute the dual for Type $\ast$ in Model 18 and Model 42.

Next, we consider the quadruple 
$$\Delta_{red}=(\GL_2\times \SO_{8},\SL_2\times \Spin_7, 1,std\otimes \Spin_7).$$
We claim that it is dual to
$$\widehat{(\Delta_{red})}=(\GL_2\times \SO_{8}, \GL_2\times \SO_4,(2^21^4),0).$$
To see this, we note that 
$$(\widehat{(\Delta_{red})})_{red}=(\GL_2\times \GL_2\times \SO_4, \GL_2\times \SO_4, 1,std\otimes std ).$$
The period integral associated with $(\widehat{(\Delta_{red})})_{red}$ is a combination of the theta correspondence of $\SO_4\times \Sp_2$ and the trilinear period $\GL_{2}^{3}/\GL_2$. Then from the period integral conjecture point of view, its dual should be given by 
$$(\GL_2\times \GL_2\times \SO_4, \GL_2\times \GL_2\times \SO_3, 1, std\otimes std \otimes \Spin_3).$$
Combining with Conjecture \ref{Whittaker induction}, we know that $\Delta_{red}$ should be dual to $\widehat{(\Delta_{red})}$. Combining  duality $\Delta_{red}\leftrightarrow \widehat{(\Delta_{red})}$ with Conjecture \ref{Whittaker induction}, we can compute the dual for Type $\ast$ in Model 33.

Lastly, we consider the quadruple 
$$\Delta_{red}=(\GL_2\times \GL_2\times \SO_{12},\SL_2\times \SO_{12}, 1,std\otimes std\oplus \HSpin_{12}).$$
We claim that it is dual to
$$\widehat{(\Delta_{red})}=(\GL_2\times \GL_2\times \SO_{12},\GL_2\times \GL_2\times \SL_2\times \SO_4 , (4^21^4),std_{\GL_2,1}\otimes std_{\GL_2,2}\otimes \HSpin_4).$$
To see this, note that $(\widehat{(\Delta_{red})})_{red}$ is given by 
$$(\GL_2\times \GL_2\times \GL_2\times \GL_2\times \SO_4,\GL_2\times \GL_2\times \SL_2\times \SO_4,1,std_{\GL_2,1}\otimes std_{\GL_2,2}\otimes \HSpin_4\oplus std_{\SL_2}\otimes std_{\SO_4}),$$ where $\SL_2$ is diagonally embedded in the 3rd and 4th copies of $\GL_2$.

The period integral associated with $(\widehat{(\Delta_{red})})_{red}$ is a combination of two theta correspondences of $\SO_4\times \Sp_2$ (this corresponds to the integral over $\GL_2\times \GL_2$ and over one $\SL_2$-copy of $\SO_4$ where the restriction of $\HSpin_4$ to it is trivial) and two trilinear periods $\GL_{2}^{3}/\GL_2$ (this corresponds to the integral over $\SL_2$ and the other $\SL_2$-copy of $\SO_4$). Hence from the period integral conjecture point of view, its dual should be given by ($\SL_2$-diagonally embeds into the 1st and 2nd copy of $\GL_2$) 
$$
(\GL_2\times \GL_2\times \GL_2\times \GL_2\times \SO_4, \SL_2\times \GL_2\times \GL_2\times \SO_4, 1, std_{\SL_2}\otimes std_{\SO_4}\oplus std_{\GL_2,1}\otimes std_{\GL_2,2}\otimes \HSpin_4).
$$
Combining this duality with Conjecture \ref{Whittaker induction}, we know that $\Delta_{red}$ should be dual to $\widehat{(\Delta_{red})}$. Combining the duality $\Delta_{red}\leftrightarrow \widehat{(\Delta_{red})}$ with Conjecture \ref{Whittaker induction}, we can compute the dual for Type $\ast$ in Model 44.

\end{document}